\documentclass[9pt,twocolumn,twoside]{IEEEtran}
\IEEEoverridecommandlockouts                        
\usepackage{times,enumerate,} 
\usepackage{color}
\usepackage{graphicx}
\usepackage{setspace}
\usepackage{bbm}
\usepackage{mathdots,mathrsfs}
\usepackage{amssymb,latexsym,amsfonts,amsmath,cite,comment,relsize}
\usepackage{stmaryrd}
\usepackage{caption}
\usepackage[dvips]{psfrag}
\usepackage{setspace}
\usepackage{amsmath}
\usepackage{url}
\usepackage{soul}
\usepackage{pgf,tikz}
\usepackage{graphicx}
\usepackage{subfig}

\newtheorem{theorem}{Theorem}
\newtheorem{lemma}{Lemma}

\newtheorem{remark}{Remark}

\newtheorem{example}{Example}
\newtheorem{assumption}{Assumption}
\newtheorem{definition}{Definition}
 \newcommand{\R}{{\mathbb{R}}}

\newcommand{\G}{{\mathcal{G}}}

\newcommand{\GG}{\mathfrak{G}}
\newcommand{\HH}{{\mathcal{H}}}
\newcommand{\V}{{\mathcal{V}}}
\newcommand{\EE}{{\mathcal{E}}}

\newcommand{\Sp}{{\hspace{0.05cm}}}
\newcommand{\LL}{{\mathfrak{L}_{n}}}

\newcommand{\rhoo}{{{\rho}}}
\newcommand{\psii}{{{\psi}}}
\newcommand{\lambdaa}{{\Lambda}}

\newcommand{\ie}{{i.e.,}}

\newcommand{\diag}{{\text{diag}}}
\usepackage{pgf,tikz}
\newcommand*{\Scale}[2][4]{\scalebox{#1}{$#2$}}
\DeclareMathOperator{\Span}{span}
\DeclareMathOperator{\cov}{{cov}}
\DeclareMathOperator{\Ker}{ker}
\DeclareMathOperator{\Dim}{dim}
\DeclareMathOperator{\Rank}{rank}
\DeclareMathOperator{\tr}{{Tr}}
\DeclareMathOperator{\E}{\mathbb E}

\begin{document}
\title{  Growing Linear Consensus Networks Endowed by \\Spectral Systemic Performance Measures}
\author{Milad Siami$^{\dagger}$  
        and~Nader Motee$^{\star}$
\thanks{$^{\dagger}$ M. Siami is with the Institute for Data, Systems, and Society, Massachusetts Institute of Technology, Cambridge, MA 02319. Email:  {\tt\small siami@mit.edu}.}
\thanks{$^{\star}$ N. Motee is with the Department of Mechanical Engineering and Mechanics, Packard Laboratory, Lehigh University, Bethlehem, PA 18015. Email:  {\tt\small motee@lehigh.edu}.}}
\maketitle

\begin{abstract}
We propose an axiomatic approach for design and performance analysis of noisy linear consensus networks by introducing a notion of systemic performance measure. This class of measures are spectral functions of Laplacian eigenvalues of the network that are monotone, convex, and orthogonally invariant with respect to the Laplacian matrix of the network. It is shown that several existing gold-standard and widely used performance  measures in the literature belong to this new class of measures. We build upon this new notion and investigate a general form of combinatorial problem of growing a linear consensus network via minimizing a given systemic performance measure.  Two efficient polynomial-time approximation algorithms are devised to tackle this network synthesis problem: a linearization-based method and a simple greedy algorithm based on rank-one updates. Several theoretical fundamental limits on the best achievable performance for the combinatorial problem is derived that assist us to evaluate optimality gaps of our proposed algorithms. A detailed complexity analysis confirms the effectiveness and viability of our algorithms to handle large-scale consensus networks.

\end{abstract}
\section{Introduction}
\allowdisplaybreaks

	{The interest in control systems society for performance and robustness analysis of large-scale dynamical network is rapidly growing \cite{Bamieh12, Siami13cdc, Young10, Bamieh11, abbas, Zelazo-Allgower, LovisariGarinZampieriResistance, Nicola, jovbamTAC05platoons, linfarjovTAC12platoons}. Improving global performance as well as robustness to external disturbances in large-scale dynamical networks are crucial for sustainability, from engineering infrastructures to living cells; examples include a group of autonomous vehicles in a formation, distributed emergency response systems, interconnected transportation networks, energy and power networks, metabolic pathways and even financial networks. One of the fundamental problems in this area is to determine to what extent uncertain exogenous inputs can steer the trajectories of a dynamical network away from its working equilibrium point. To tackle this issue, the primary challenge is to introduce meaningful and viable performance and robustness measures that can capture essential characteristics of the network.  A proper measure should be able to encapsulate transient, steady-state, macroscopic, and microscopic features of the perturbed large-scale dynamical network.   

In this paper, we propose a new methodology to classify several proper performance measures for a class of linear consensus networks subject to external stochastic disturbances. We take an axiomatic approach to quantify essential functional properties of a number of sensible measures  by introducing the class of systemic performance measures and show that this class of measures should satisfy monotonicity, convexity, and orthogonal invariance properties. It is shown that several existing and widely used performance measures in the literature are in fact special cases of this class of systemic measures \cite{Siami14acc, Zelazo-Allgower,Young10,noami11,Jadbabaie13}. 

	The performance analysis of linear consensus networks subject to external stochastic disturbances has been studied in \cite{Bamieh12, Siami13cdc, Siami14arxiv, SiamiNecSys, noami11, Jadbabaie13, dorjovchebulTPS14}, where the $\mathcal H_2$-norm of the network was employed as a scalar performance measure. In \cite{Bamieh12}, the authors interpret the $\mathcal H_2$-norm of the system as a macroscopic performance measure capturing the notion of coherence. It has been shown that if the Laplacian matrix of the coupling graph of the network is normal, the $\mathcal H_2$-norm is a function of the eigenvalues of the Laplacian matrix \cite{noami11,Bamieh12, SiamiNecSys}. In \cite{Siami13cdc}, the authors consider general linear dynamical networks and show that tight lower, and upper bounds can be obtained for the $\mathcal H_2$-norm of the network from the exogenous disturbance input to a performance output, which are functions of the eigenvalues of the state matrix of the network. Besides the commonly used $\mathcal H_2$-norm, there are several other performance measures that have been proposed in \cite{Bamieh12, Zelazo-Allgower, Olfati-saber07}.   In \cite{Siami14acc}, a partial ordering on linear consensus networks is introduced where it shows that several previously used performance measures are indeed Schur-convex functions in terms of the Laplacian eigenvalues. In a more relevant work, the authors of \cite{Siami14cdc-2} show that  performance measures that are defined based on some system norms, spectral, and entropy functions  exhibit several useful functional properties  that allow us to utilize them in network synthesis problems.

The first main contribution of this paper is introduction of a class of systemic performance measures that are spectral functions of Laplacian eigenvalues of the coupling graph of a linear consensus network. Several gold-standard and widely used performance measures belong to this class, for example, to name only a few, spectral zeta function, Gamma entropy, expected transient output covariance, system Hankel norm, convergence rate to consensus state, logarithm of uncertainty volume of the output, Hardy-Schatten system norm or $\mathcal{H}_p$-norm, and many more. All these performance measures are monotone, convex, and orthogonally invariant. Our main goal is to investigate a canonical network synthesis problem: growing a linear consensus network by adding new interconnection links to the coupling graph of the network and minimizing a given systemic performance measure. In the context of graph theory, it is known that a simpler version of this combinatorial problem, when the cost function is the inverse of algebraic connectivity, is indeed NP-hard \cite{Mosk}. There have been some prior attempts to tackle this problem for some specific choices of  cost functions (i.e., total effective resistance and the inverse of algebraic connectivity) based on semidefinite programing (SDP)  relaxation methods  \cite{Kolla, Ghosh2006}. There is a similar version of this problem that is reported in \cite{Fardad}, where the author studies convergence rate of circulant  consensus networks by adding some long-range links. Moreover, a continuous (non-combinatorial)  and relaxed version of our problem of interest has some connections to the sparse consensus network design problem \cite{mogjovACC15, wujovACC14, farlinjovTAC14sync}, where they consider $\ell_1$-regularized $\mathcal H_2$-optimal control problems. The other related works \cite{Summers16, SummersECC} argue that some metrics based on controllability and observability Gramians are modular or submodular set functions, where they aim to show that their proposed simple greedy heuristic algorithms have   guarantees sub-optimality bounds. 

In our second main contribution, we propose two efficient  polynomial-time  approximation algorithms to solve the above mentioned combinatorial network synthesis problem: a linearization-based method and a simple greedy algorithm based on rank-one updates. Our complexity analysis asserts that computational complexity of our proposed algorithms are reasonable and make them particularly suitable for synthesis of large-scale consensus networks. To calculate sub-optimality gaps of our proposed  approximation algorithms, we quantify the best achievable performance bounds for the network synthesis problem in Section  \ref{sec:672}. Our obtained fundamental limits are exceptionally useful as they only depend on the spectrum of the original network and they can be computed a priori.  In Subsection \ref{subsec1}, we classify a subclass of differentiable systemic performance measures that are indeed supermodular. For this subclass, we show that our proposed simple greedy algorithm can achieve a $(1- 1/e)$-approximation of the optimal solution of the combinatorial network synthesis problem. Our extensive simulation results confirm effectiveness of our proposed methods.


\section{Preliminaries and Definitions}
\label{sec:123}
\allowdisplaybreaks

\subsection{Mathematical Background}

	The set of real numbers is denoted by $\R$, the set of non--negative by $\R _{+}$, and the set of positive real numbers by $\R _{++}$. The cardinality of set $\EE$ is shown by $|\EE|$. We assume that $\mathbbm{1}_n$, $I_n$, and $J_n$ denote the $n \times 1$ vector of all ones, the $n \times n$ identity matrix, and the $n \times n$ matrix of all ones, respectively. For a vector $v = [v_i] \in \mathbb R^n$, $\diag(v) \in \R^{n \times n}$ is the diagonal matrix with elements of $v$ orderly sitting on its diameter, and for $ A= [a_{ij}] \in \mathbb R^{n \times n}$, $\diag(A) \in \R^{n}$ is diagonal elements of square matrix $A$. We denote the generalized matrix inequality with respect to the positive semidefinite cone $\mathbb{S}^n_{+}$ by ``$\,\preceq \,$" .

{Throughout this paper, it is assumed that all graphs are finite, simple, undirected, and connected.}  A   graph herein is defined by a triple $\G = (\V, \EE,w)$, where $\V$ is the set of nodes, $\EE \subseteq \big\{\{i,j\}~\big|~ i,j \in \V, ~i \neq j \big\}$ is the set of links, and $w: \EE \rightarrow  \R_{++}$ is the weight function. 
{The adjacency matrix $A = [a_{ij}]$ of graph $\G$ is defined in such a way that $a_{ij} = w(e)$ if $e=\{i,j\} \in \EE$, and $a_{ij}=0$ otherwise. The Laplacian matrix of $\G$ is defined by $L := \Delta - A$, where $\Delta=\diag[d_{1},\ldots,d_{n}]$ and $d_i$ is degree of node $i$.}  We denote the set of Laplacian matrices of all connected weighted graphs with $n$ nodes by $\LL$. Since $\G$ is both undirected and connected, the Laplacian matrix $L$ has $n-1$ strictly positive eigenvalues and one zero eigenvalue. Assuming that $0 = \lambda_1 < \lambda_2 \leq \ldots \leq \lambda_n$ are eigenvalues of Laplacian matrix $L$, we define operator ${\lambdaa}: \mathbb{S}^n_{+} \rightarrow \R^{n-1}_{++}$ by  
\begin{equation} 
	\lambdaa(L) ~=~ \begin{bmatrix} \lambda_2 & \ldots & \lambda_n \end{bmatrix}^{\text T}. \label{eigen-fcn}
\end{equation}
The Moore-Penrose pseudo-inverse of $L$ is denoted by $L^{\dag}=[l_{ji}^{\dag}]$, which is a square, symmetric, doubly-centered and positive semi--definite matrix. For a given link $e=\{i,j\}$, $r_e(L)$ denotes the effective resistance between nodes $i$ and $j$ in a graph with the Laplacian matrix $L$, where its value can be calculated as follows
\begin{equation}
	r_{e}(L)~=~l_{ii}^{\dag}+l_{jj}^{\dag}-2 \hspace{0.02cm} l_{ij}^{\dag},
	\label{eq:148}
\end{equation}	
where $L^{\dag}=[l_{ij}^{\dag}]$. For every real $q$, powers of pseudo inverse of $L$ is represented by $L^{\dag,q} := \left(L^{\dag} \right)^q.$

\begin{definition}
The derivative of a scalar function $\rho(.)$, with respect to the $n$-by-$n$ matrix $X$, is defined by
\[\triangledown \rho(X) ~:=	\left[\begin{array}{cccc} 
		\frac{\partial \rho}{\partial x_{11}} & \frac{\partial \rho}{\partial x_{12}}    & \ldots & \frac{\partial \rho}{\partial x_{1n}} \\
		\frac{\partial \rho}{\partial x_{21}}  & \frac{\partial \rho}{\partial x_{22}} & \ldots & \frac{\partial \rho}{\partial x_{2n}}    \\
		\vdots & \vdots  & \ddots &\vdots    \\
		\frac{\partial \rho}{\partial x_{n1}} & \frac{\partial \rho}{\partial x_{n2}} & \ldots & \frac{\partial \rho}{\partial x_{nn}}
		\end{array}\right] \label{state-matrix},\]
where $X=[x_{ij}]$. The directional derivative of function $\rho(X)$ in the direction of matrix $Y$ is given by
\[\triangledown_{Y} \rho(X)~=~~\big < \triangledown \rhoo(X) , Y \big> ~ =~\tr \left(\triangledown \rhoo(X) Y\right),\]
where $\left < .,. \right > $ denotes the inner product operator. 
\end{definition}

The following Majorization definition is from \cite{marshall11}.  
\begin{definition}
For every $x \in \R_+^n$, let us define $x^{\downarrow}$ to be a vector whose elements are a permuted version of elements of $x$ in descending order. We say that $x$  majorizes $y$, which is denoted by $x \unrhd y$, if and only if $\mathbf{1}^{\text T}x \, = \, \mathbf{1}^{\text T}y$ and $\sum_{i=1}^k x_i^{\downarrow} \, \geq \, \sum_{i=1}^k y_i^{\downarrow}$ for all $k=1,\ldots,n-1$.
\end{definition}

The vector majorization is not a partial ordering. This is because from relations $x \unrhd y$ and $y \unrhd x$ one can only conclude that the entries of these two vectors are equal, but possibly with different orders. Therefore, relations $x \, \unrhd \, y$ and $y \, \unrhd \, x$ do not imply $x=y$.  

\begin{definition}[\cite{marshall11}]
The real-valued function $F: \R_+^n \rightarrow \R$ is called Schur--convex if $F(x) \geq F(y)$ for every two vectors  $x$ and $y$ with property $x \unrhd y$. 
\end{definition}

{

\subsection{Noisy linear consensus networks } 
\label{sec:158}

	We consider the class of  linear dynamical networks that consist of multiple agents with scalar state variables $x_i$ and control inputs $u_i$ whose dynamics evolve in time according to 	
\begin{eqnarray}
\dot{x}_i(t) & = & u_i(t) +\xi_i(t) \label{TI-consensus-algorithm} \\
y_i(t) & = & x_i(t) - \bar{x}(t)  \label{TI-consensus-algorithm-2}
\end{eqnarray}
for all $i=1,\ldots,n$, where  $x_i(0)=x_i^*$ is the initial condition and \[\bar{x}(t)=\frac{1}{n}\big(x_1(t)+\ldots+x_n(t)\big)\] 
is the average of all states at time instant $t$. The impact of the uncertain environment on each agent's dynamics is modeled by the exogenous noise input $\xi_i(t)$. 
By applying the following feedback control law to the agents of this network 
\begin{equation}
u_i(t) ~=~\sum_{j=1}^{n} k_{ij} \big(x_j(t) - x_i(t)\big),\label{feedback-law}
\end{equation}
the resulting closed-loop system will be a first-order linear consensus network. The closed-loop dynamics of network (\ref{TI-consensus-algorithm})-\eqref{TI-consensus-algorithm-2} with feedback control law \eqref{feedback-law} can be written in the following compact form
\begin{eqnarray}
	\dot x(t) & = &  -L\, x(t)~+~\xi(t)\label{first-order}\\
	y(t) & = & M_n \, x(t), \label{first-order-G}
\end{eqnarray}
%
with  initial condition $x(0)=  x^*$, where  $x = [x_1,  \ldots,  x_n]^{\rm T}$ is the state, $y = [y_1,  \ldots,  y_n]^{\rm T}$ is the output, and $\xi = [\xi_1,  \ldots,  \xi_n]^{\rm T}$ is the exogenous noise input of the network. 
The state matrix of the network is a graph Laplacian matrix that is defined by $L=[l_{ij}]$, where 
\begin{equation}
\displaystyle l_{ij} := \left\{\begin{array}{ccc}
-k_{ij} & \textrm{if} & i \neq j \\
 &  &  \\
k_{i1}+\ldots+k_{in}& \textrm{if} & i=j
\end{array}\right.
\end{equation}
and the output matrix is a  centering matrix that is defined by
\begin{equation}
M_n~:=~I_{n} - \frac{1}{n}J_n. 
\end{equation}

The underlying coupling graph of the consensus  network \eqref{first-order}-\eqref{first-order-G} is a graph $\G=(\V,\mathcal E, w)$ with node set $\V=\{1,\ldots,n\}$, edge set 
\begin{equation} 
	\EE=\Big\{ \{i,j\}~\big|~\forall~i,j \in \V,~k_{ij} \neq 0\Big\}, \label{edge-set}
\end{equation}
and weight function 
\begin{equation}
	w(e)=k_{ij}   \label{edge-weight}
\end{equation}
for all $e=\{i,j\} \in \EE$, and $w(e)=0$ if $e \notin \EE$. The Laplacian matrix of graph $\G$ is equal to $L$. 
\begin{assumption}\label{assump-simple}
All feedback gains (weights) satisfy the following properties for all $i,j \in \V$: 

\vspace{0.1cm}
\noindent (a)~non-negativity: $k_{ij} \geq 0$, \\
\noindent (b)~symmetry: $k_{ij}=k_{ji}$,\\
\noindent (c)~simpleness: $k_{ii}= 0$.
\vspace{0.1cm}
\end{assumption}

Property (b) implies that feedback gains are symmetric and (c) means that there is no self-feedback loop in the network.

\vspace{0.1cm}
\begin{assumption}\label{assum-coupling-graph}
The coupling graph $\G$ of the consensus network \eqref{first-order}-\eqref{first-order-G} is connected and time-invariant.
\end{assumption}
\vspace{0.1cm}

According to Assumption \ref{assump-simple}, the underlying coupling graph is undirected and simple. Assumption \ref{assum-coupling-graph} implies that only one of the modes of network \eqref{first-order} is marginally stable  with eigenvector $\mathbbm{1}_n$ and all other ones are stable. The marginally stable mode, which corresponds to the only zero Laplacian eigenvalue of $L$, is unobservable from the output \eqref{first-order-G}. The reason is that the output matrix of the network satisfies $M_n \mathbbm{1}_n= 0$. When there is no exogenous noise input, i.e., $\xi(t) =  0$ for all time, state of all agents converge to a consensus state  \cite{Olfati-saber07, Jadbabaie}, which for our case the consensus state is 
\begin{equation}
\lim_{t \rightarrow \infty} x(t) ~=~ \bar x(0) \mathbbm{1}_n~=~\frac{1}{n}\mathbbm{1}_n\mathbbm{1}_n^{\text T} x^*. \label{limit-zero} %
\end{equation}
When the network is fed with a nonzero exogenous noise input, the limit behavior \eqref{limit-zero} is not expected anymore and the state of all agents will be fluctuating around the consensus state without converging to it. Before providing a formal statement of the problem of growing a linear consensus network, we need to introduce a new class of performance measures for networks \eqref{first-order}-\eqref{first-order-G} that can capture the effect of noise propagation throughout the network and   quantify degrees to which the state of all agents are dispersed from the consensus state.

\section{Systemic Performance Measures} 
\label{sec: 176}

	The notion of systemic performance measure  refers to a real-valued operator over the set of all linear consensus networks governed by \eqref{first-order}-\eqref{first-order-G} with the purpose of quantifying the quality of noise propagation in these networks. 
We have adopted an axiomatic approach to introduce and categorize a class of such operators that are obtained through our close examination of functional properties of several existing gold standard measures of performance in the context of network engineering and science. In order to state our findings in a formal setting, we observe that every network with dynamics  \eqref{first-order}-\eqref{first-order-G} is uniquely determined by its Laplacian matrix. Therefore, it is reasonable to define a systemic performance measure as an operator over the set of Laplacian matrices $\LL$.	
	
\begin{definition}\label{def-schur-systemic} 
An operator $\rhoo: \LL \rightarrow \R$ is called a systemic performance measure if it satisfies the following properties for all Laplacian matrices in $\LL$:
{\vspace{0.0cm}

\noindent 1. {\it Monotonicity:} If $L_{2} \preceq L_{1}$, then
	\[\rhoo (L_{1}) ~\leq~ \rhoo (L_{2});\]	
	
\vspace{0.05cm}

\noindent 2. {\it Convexity:} For all $0 \leq \alpha \leq 1$,
	\[\rhoo (\alpha L_{1}+(1-\alpha)L_{2})~\leq~ \alpha \rhoo (L_{1})+{(1-\alpha)}\rhoo (L_{2});\]

\vspace{0.05cm}

\noindent 3. {\it Orthogonal invariance:} For all orthogonal matrices $U \in \R^{n \times n}$,  
\[\rhoo(L) ~=~ \rhoo(U L U^{\text T}).\]}
\end{definition}
\vspace{0.0cm}

Property 1 guarantees that strengthening  couplings in a consensus network never worsens the network performance with respect to a given systemic performance measure. The coupling strength among the agents can be enhanced by several means, for example, by adding new feedback interconnections and/or increasing weight  of an individual feedback interconnection. The monotonicity property induces a partial ordering\footnote{This implies that the family of networks \eqref{first-order}-\eqref{first-order-G} can be ordered using a relation that has reflexivity, antisymmetry, and transitivity properties.}  on all linear consensus networks governed by \eqref{first-order}-\eqref{first-order-G}. Property 2 requires that a viable performance measure should be amenable to convex optimization algorithms for network synthesis purposes. Property 3 implies that a systemic performance measure depends only on the Laplacian eigenvalues. 

\begin{theorem}\label{thm:schur-convex}
Every operator $\rhoo: \LL \rightarrow \R$ that satisfies Properties 2 and 3 in Definition \ref{def-schur-systemic} is indeed a Schur-convex function of Laplacian eigenvalues, i.e., there exists a Schur-convex spectral function $\Phi: \R^{n-1} \rightarrow \R$ such that 
\begin{equation}
	\rhoo(L)~=~ \Phi(\lambda_2, \ldots, \lambda_n). \label{spectral-rho} 
\end{equation}
\end{theorem}
\begin{proof}
For every $L \in \LL$, the value of the systemic performance measure can be written as a composition of two functions as follows
\begin{equation}
	\rhoo (L) ~=~ (\phi \circ \Lambda)(L), \label{spectral-fcn}
\end{equation}
where function $\Lambda: \mathbb{S}^n_{+} \rightarrow \R^{n-1}_{++}$ is defined by \eqref{eigen-fcn} and function $\phi: \mathbb R_{++}^{n-1} \rightarrow \R$ is characterized by
\begin{equation} 
	\phi(v)~=~\rhoo(W^{\text T} \textrm{diag}(v) W)
	\label{eq:250}
\end{equation}
for any matrix  $W=EU$ with $U \in \R^{n \times n}$ being an orthogonal matrix satisfying $L=U^{\text T} \textrm{diag}([0,\lambdaa(L)^{\text T}]) U$ and $E \in \R^{(n-1) \times n}$ given by the following projection matrix 
\begin{equation} 
E ~=~ \left[\begin{array}{ccc}
0_{(n-1) \times 1} & \big| & I_{n-1} \end{array} 
\right].  
\end{equation}
Thus, we can conclude that \eqref{spectral-rho} holds with $\Phi(\lambda_2, \ldots, \lambda_n)=\phi(\Lambda(L))$. In the next step, we need to show that operator $\rhoo$ is convex and symmetric with respect to Laplacian eigenvalues $\lambda_2, \ldots, \lambda_n$. Property 2 indicates that $\rhoo$ is convex on Laplacian matrices and any convex function on Laplacian matrices is also convex function with respect to Laplacian eigenvalues \cite{boyd2006}. Property 3 implies that operator $\rhoo$ is symmetric with respect to $\lambda_2, \ldots, \lambda_n$ as $\rhoo$ is invariant under any permutation of eigenvalues. It is known that every function that is convex and symmetric is also Schur-convex \cite{boyd2006}. 
\end{proof}

\begin{table*}[t]
{\small
\begin{center}
    \begin{tabular}{ | p{6.5cm}  |  p{9cm} |}
    \hline
Systemic Performance Measure & Matrix Operator Form \\ \hline \hline
 Spectral zeta function ${\zeta}_{q}(L)$ &  $\left(\mathrm{Tr}\big( L^{\dagger,q}  \big)\right)^{\frac{1}{q}}$   \vspace{0.1cm} \\
\hline   
Gamma entropy $I_{\gamma}(L)$ &    $\displaystyle  \gamma^2 \mathrm{Tr}\Big(L - \big( L^2 - \gamma^{-2} M_n \big)^{\frac{1}{2}} \Big)$ 
  \vspace{0.1cm}
        \\
    \hline
 Expected  transient output covariance $\tau_t(L)$ & $\displaystyle  \frac{1}{2} \mathrm{Tr}\big(L^{\dagger} (I- e^{-L t})\big)$  
   \vspace{0.1cm}
        \\
    \hline    
System Hankel norm  $\eta(L)$ & $\displaystyle  \frac{1}{2} \max \big\{ \mathrm{Tr}(L^{\dagger} X)~\big|~X=X^{\rm T},~ \mathrm{rank}(X)=1, ~\mathrm{Tr}(X)=1 \big\} 
   \vspace{0.05cm}
$  \\
    \hline  
      \vspace{0.05cm}
Uncertainty volume of the output   $\upsilon(L)$ & \vspace{0.05cm} $\displaystyle  (1-n) \log 2 - \mathrm{Tr}\left(\log \left(L+\frac{1}{n} J_n\right)\right)$ 
  \vspace{0.1cm}
        \\
    \hline  
 Hardy-Schatten system norm or $\mathcal{H}_p$-norm $\theta_p(L)$ & $\displaystyle \alpha_0 \left( \tr \left( L^{\dag,\,p-1}\right)\right)^{\frac{1}{p}}$ 
        \\
    \hline   
    \end{tabular}
        \caption{ \small Some important examples of spectral systemic performance measures and their corresponding matrix operator forms. }
        \label{matrix-operator}
\end{center}}
\vspace{-0.6cm}
\end{table*}

	The Laplacian eigenvalues of network \eqref{first-order}-\eqref{first-order-G} depend on global features of the underlying coupling graph. This is the reason why every performance measure that satisfies Definition \ref{def-schur-systemic} is tagged with adjective {\it systemic}.  Table  \ref{matrix-operator} shows some important examples of systemic performance measures and their corresponding matrix operator forms. In the appendix, we prove functional properties  and discuss  applications of these measures in details.  

\section{Growing a Linear Consensus Network}
\label{sec:218}

	The network synthesis problem of interest is  to improve the systemic performance of network  \eqref{first-order}-\eqref{first-order-G} by establishing $k \geq 1$ new feedback interconnections among the agents. Suppose that the underlying graph of the network $\G=(\V,\mathcal E, w)$ is defined according to \eqref{edge-set}-\eqref{edge-weight} and  a set of candidate feedback interconnection links $\mathcal E_c=\big\{\varepsilon_1, \ldots , \varepsilon_p \big\} {\, \subseteq \,} \V \times \V$, which is endowed with a weight function $\varpi: \EE_c \rightarrow \R_{++}$, is also given. The weight of a link $\varepsilon_i \in \mathcal E_c$ is represented by $\varpi(\varepsilon_i)$ and we assume that it is pre-specified and fixed. The network growing problem is to select exactly $k$ feedback interconnection links from $\mathcal E_c$ and append them to $\G$ such that the systemic performance measure of the resulting network is minimized over all possible choices. 

Let us represent the set of all  possible appended subgraphs by
\begin{equation*}
	\hat{\mathfrak{G}}_k:=\Big\{ \hat{\G} = (\mathcal V, \hat{\EE}, \hat{w})\Sp\Big|\Sp\hat{\EE} \in \Pi_k(\EE_c),~\forall \varepsilon_i \in \hat{\EE}:~ \hat{w}(\varepsilon_i)= 	\varpi(\varepsilon_i) \Big\},
\end{equation*}
where the set of all possible choices to select $k$  links is denoted by  
\begin{equation*}
	\Pi_k(\EE_c) := \big\{ \hat{\EE} \subseteq \mathcal E_c~\big|~|\hat{\EE}|=k\big\}.
\end{equation*}
Then, the network synthesis problem can be cast as the following combinatorial optimization problem
\begin{equation}
	\underset{\hat{\G} \in \hat{\mathfrak{G}}_k}{\textrm{minimize}}   \hspace{0.6cm} \rhoo (L + \hat{L}), \label{k-link}
\end{equation}
where $\hat{L}$ is the Laplacian matrix of an appended candidate subgraph $\hat{\G}$  and the resulting network with Laplacian matrix $L+\hat{L}$ is referred to as the augmented network. 
The role of the candidate set $\EE_c$ is to pre-specify authorized  locations to establish new feedback interconnections in the network. 

	The network synthesis problem \eqref{k-link} is inherently combinatorial and it is known that a simpler version of this problem with $\rhoo(L)=\lambda^{-1}_2$ is in fact NP-hard \cite{Mosk}.  There have been some prior attempts to tackle problem \eqref{k-link} for some specific choices of   performance measures, such as total effective resistance and the inverse of algebraic connectivity, based on convex relaxation methods \cite{Kolla, Ghosh2006} and greedy methods \cite{SummersECC}. In Sections \ref{subsec:B} and \ref{sec:algorithms}, we propose approximation algorithms to compute sub-optimal solutions for  \eqref{k-link} with respect to the broad class of systemic performance measures.  
We propose an exact solution for \eqref{k-link} when $k=1$ and two tractable and efficient approximation methods when $k >1$ with computable performance bounds. Besides, in Section \ref{sec:algorithms}, we demonstrate that a subclass of systemic performance measures has a supermodularity property. This provides approximation guarantees for our proposed approximation algorithm.

\section{Fundamental Limits on the Best Achievable Performance Bounds}
\label{sec:672}
In the following, we present theoretical bounds for the best achievable values for the performance measure in  \eqref{k-link}. Let us denote the optimal cost value of the optimization problem \eqref{k-link} by $\mathbf{r}_k^*(\varpi)$.
 
For a given systemic performance measure $\rhoo: \LL \rightarrow \R$, we recall that  according to Theorem \ref{thm:schur-convex} there exists a spectral function $\Phi$ such that 
\[ \rhoo(L)~=~ \Phi \big(\lambda_2, \ldots, \lambda_n\big). \]

\begin{theorem}\label{w-thm}
Suppose that a consensus network \eqref{first-order}-\eqref{first-order-G} with an ordered set of Laplacian eigenvalues $\lambda_2 \leq \ldots \leq \lambda_n$, a set of candidate links $\EE_c$ endowed with a weight function $\varpi: \EE_c \rightarrow \R_{++}$, and design parameter $1 \leq k \leq n-1$ are given.  Then, the following inequality 
\begin{equation}
		\mathbf{r}_k^*(\varpi) ~>~    \Phi \big (\lambda_{k+2}, \ldots, \lambda_n,\underbrace{ \infty, \ldots, \infty}_{\text{$k$ times}}\big )
		\label{fund-limit-1}
\end{equation}
holds for all  weight functions $\varpi$. For $k \geq n$, all lower bounds are equal to  $\Phi \big (\infty, \ldots, \infty\big )$. Moreover, if the systemic performance measure has the following decomposable form
	\begin{equation*}
		\rhoo \left(L\right)~=~\sum_{i=2}^n \varphi (\lambda_i),
		\label{meas_0}
	\end{equation*}
where $\varphi: \R \rightarrow \R_{+}$ is a  decreasing convex function and $\lim_{\lambda \rightarrow \infty} \varphi(\lambda) = 0$, then the best achievable performance measure is characterized  by
\begin{equation}
		\mathbf{r}_k^*(\varpi)~ > ~\sum_{i=k+2}^{n}\varphi(\lambda_i). \label{lower-bound-rho_0}
\end{equation} 

\end{theorem}

\begin{proof}
For a given weight function $\varpi: \EE_c \rightarrow \R_{++}$, we show that inequality (\ref{fund-limit-1}) holds for every $\hat{\EE} \in \Pi_k(\EE_c)$. Assume that $\hat L$ is  the Laplacian of the graph formed by $k$ added edges. We note that ${\Rank}( \hat L)=k' \leq k$. Therefore ${\Dim}({\Ker}\hat L)=n-k' \geq n-k$. Therefore, we can define the nonempty set $M_j$ for $2\leq j \leq n$, as follows
	\begin{equation*}
		M_j \,=\, {\Span}\{u_1,\ldots,u_{j+k'}\} \, \cap \, \Span\{v_j,\ldots,v_{n}\} \, \cap \, \Ker \hat L,
	\end{equation*}
where $u_i$'s and $v_i$'s are orthonormal eigenvectors of  $L$ and $L+\hat L$, respectively. We now choose a unit vector $v \in M_j$. It  then follows that:
	\begin{eqnarray}
		\lambda_j (L+ \hat L) &\leq& v^{\text T}( L+ \hat L)v~=~v^{\text T}Lv \nonumber \\
		&\leq& \lambda_{j+k'}( L)~\leq~ \lambda_{j+k}( L).
	\label{eq-h}
	\end{eqnarray}
Therefore, according to \eqref{eq-h} and the monotonicity property of the systemic measure $\rhoo$, we get 
	\begin{eqnarray}
		\rhoo(L+\hat L) ~>~   \Phi \big (\lambda_{k+2}, \cdots, \lambda_n,\underbrace{ \infty, \cdots, \infty}_{\text{$k$ times}}\big ),
		\label{eq:369}
	\end{eqnarray}	
for all $\hat{\EE} \in \Pi_k(\EE_c)$. Inequality (\ref{fund-limit-1}) now follows from (\ref{eq:369}) and this completes the proof. Note that inequality \eqref{lower-bound-rho_0} is a direct consequence of \eqref{fund-limit-1} and $\lim_{\lambda \rightarrow \infty} \varphi(\lambda) = 0$.
\end{proof}

\begin{theorem}		
\label{w-prop}
Suppose that in optimization problem \eqref{k-link}, the set of candidate links form a complete graph, i.e., $|\mathcal E_c|=\frac{1}{2}n(n-1)$. Then, there exists a weight function $\varpi_0: \EE_c \rightarrow \R_{++}$ and a choice of $k$ weighted links from $\EE_c$ with weight function $\varpi: \EE_c \rightarrow \R_{++}$ such  that 
	\begin{equation}
		\mathbf{r}_k^*(\varpi)~\leq~  \Phi \big (\lambda_{2}, \ldots, \lambda_{n-k}, \underbrace{ \infty, \ldots, \infty}_{\text{$k$ times}} \big )
		\label{fund-limit-2}
		\end{equation}
holds for all weight functions $\varpi$ that satisfies $\varpi(e) \geq \varpi_0(e)$ for all $e \in \EE_c$. Moreover, if the systemic performance measure has the following decomposable form
	\begin{equation*}
		\rhoo \left(L\right)~=~\sum_{i=2}^n \varphi (\lambda_i),
		\label{meas}
	\end{equation*}
where $\varphi: \R \rightarrow \R_{+}$ is a  decreasing convex function and $\lim_{\lambda \rightarrow \infty} \varphi(\lambda) = 0$, then the best achievable performance measure is characterized  by
\begin{equation}
		\mathbf{r}_k^*(\varpi)~ \leq ~\sum_{i=2}^{n-k}\varphi(\lambda_i). \label{lower-bound-rho}
\end{equation} 
\end{theorem}

\begin{proof}
{We will show that there exists $\hat{\EE} \in \Pi_k(\EE_c)$ for which (\ref{fund-limit-2}) is satisfied. Without loss of generality, we may assume that $k < n-1$. This is because otherwise, by adding $n-1$ links, which forms a spanning tree, and increasing their weights the performance of the resulting network tends to $\Phi(\infty, \cdots, \infty)$ {(see Theorem \ref{tree-theorem})}. Let $\hat \EE \subset \EE$ be the set of $k$ links that do not form any cycle with $\varpi_0(e)=\infty$ for all $e \in \hat \EE$. Then, we know that
	\begin{equation}
		\lambdaa(L + \hat L) ~\geq~ \lambdaa(L)
	\label{fund-limit-2-1}
	\end{equation}
and the $k$ largest eigenvalues of $L + \hat L$ are equal to $\infty$. Using \eqref{fund-limit-2-1} and the monotonicity property of the systemic performance measure, we get 
	\begin{equation}
		\rhoo(L+\hat L)~\leq~  \Phi \big (\lambda_{2}, \cdots, \lambda_{n-k}, \underbrace{ \infty, \cdots, \infty}_{\text{$k$ times}} \big ).
		\label{fund-limit-22}
	\end{equation}
From $\mathbf{r}^*(\varpi) \leq \rhoo(L+\hat L)$ and using \eqref{fund-limit-22}, we obtain \eqref{fund-limit-2}. Note that inequality \eqref{lower-bound-rho} is a direct consequence of \eqref{fund-limit-2} and $\lim_{\lambda \rightarrow \infty} \varphi(\lambda) = 0$.
} 
\end{proof}

Examples of systemic performance measures that satisfies conditions of Theorem \ref{w-thm} include $\zeta_q^q(L)$ for $q \geq 1$, $I_\gamma(L),$ and $\tau_t(L)$. 

\begin{theorem}
\label{tree-theorem}
Let us consider a linear consensus network \eqref{first-order}-\eqref{first-order-G} that is endowed with systemic performance measure 	$\rho: \LL \rightarrow \R$. Then,  the network performance can be arbitrarily improved\footnote{This implies that the value of the systemic performance measure can be made close enough to $\Phi(\infty, \cdots, \infty)$, the lower bound in inequality \eqref{fund-limit-1}.} by adding only $n-1$ links that form a spanning tree.
\end{theorem}

\begin{proof}
Let us denote the Laplacian matrix of the spanning tree by $L_{\mathcal T}$. In the following, we show that the performance of resulting network can be arbitrarily improved by increasing the weights of the spanning trees.
Based on the monotonicity property, we have 
\begin{equation}
\rho(L+\kappa L_{\mathcal T}) \leq \rho ( \kappa L_{\mathcal T}), ~~~{\kappa >0},
\label{eq:1512}
\end{equation}
Also, we know that $\Lambda(\kappa L_{\mathcal T}) = \kappa \Lambda(L_{\mathcal T})$. Therefore, using the fact that the spanning tree has only one zero eigenvalue, \eqref{spectral-rho}, we get
\[ \lim_{\kappa \rightarrow \infty }   \rho ( \kappa L_{\mathcal T}) =  \Phi (\infty, \cdots, \infty). \]
Using this limit and \eqref{eq:1512} we get the desired result.
\end{proof}

It should be emphasized  that by increasing weights of all the edges, the network performance can be arbitrarily improved, \ie~  the value of the systemic performance measure can be made arbitrarily close to $\Phi(\infty, \cdots, \infty)$. Theorem \ref{tree-theorem} sheds more light on this fact by revealing the minimum number of  required links and their graphical topology to achieve this goal. 

The results of Theorems \ref{w-thm} and \ref{w-prop} can be effectively applied to select a suitable value for the design parameter $k$ in optimization problem \eqref{k-link}.  Let us denote the value of the lower bound in \eqref{fund-limit-1} by $\varrho_k$. The performance of the original network is then $\varrho_0=\rho(L)$. The percentage of performance enhancement can be computed by formula $\frac{\varrho_0 - \varrho_k}{\varrho_0} \times 100$ for all values of parameter $1 \leq k \leq n-1$. For a given desired performance level, we can look up these numbers and find the minimum number of required  links to be added to the network.  This is explained in details in Example \ref{ex:4} and Figure \ref{fig:882} in Section \ref{sec:simu}. In next sections, we propose approximation algorithms to compute near-optimal solutions for the network synthesis problem \eqref{k-link}. 

\section{A Linearization-Based Approximation Method} \label{sec:linearization}
\label{subsec:B}
Our first approach is based on a linear approximation of the systemic performance measure when weights of the candidate links in $\EE_c$ are small enough. In the next result, we calculate Taylor expansion of a systemic performance measure using notions of directional derivative for spectral functions. 

\begin{lemma} 
\label{second-approx}
Suppose that a linear consensus network  \eqref{first-order}-\eqref{first-order-G} endowed with a differentiable systemic performance measure $\rho$ is given. Let us consider the cost function in optimization problem \eqref{k-link}. If $\hat L$ is the Laplacian matrix of an appended subgraph $\hat{\G} = (\mathcal V, \hat{\EE}, \varpi)$, then  
\begin{equation*}
{\rhoo(L+\epsilon \hat L) ~=~ \rhoo (L) + \epsilon \tr \big( \triangledown \rhoo(L) \hat L \big) +\mathcal{O}(\epsilon^2)}
\end{equation*}
where the derivative of $\rho$ at $L$ is given by 
\begin{equation}
		\triangledown \rhoo(L) ~=~ W^{\text T} \left( \diag \triangledown \phi \left (\lambdaa(L)\right ) \right) W \label{derivative-rho}
\end{equation}
for any matrix $W$ that is defined by \eqref{eq:250}.
\end{lemma}

\begin{proof}
The expression \eqref{derivative-rho} can be calculated using the spectral form  of a given systemic performance measure described by  \eqref{spectral-fcn} and according to  \cite[Corollary 5.2.7]{borwein}. Using the directional derivative of $\rhoo$ along matrix $\hat L$, the Taylor expansion of $\rhoo(L+\epsilon \hat L) $ is given by
	\begin{equation}
		\rhoo(L+ \epsilon \hat L) ~=~ \rhoo(L) + \epsilon \triangledown_{\hat L} \rhoo (L)  +\mathcal{O}(\epsilon^2),
		\label{1182}
	\end{equation}
where $ \triangledown_{\hat L} \rhoo (L) $ is the directional derivative of $\rhoo$ at $L$ along matrix $\hat L$
	\begin{equation}
		\triangledown_{\hat L} \rhoo (L)~=~\big < \triangledown \rhoo(L) , \hat L \big> ~=~\tr \big( \triangledown \rhoo(L) \hat L \big),
		\label{d-d}
	\end{equation}
where $\left < .,. \right > $ denotes the inner product operator. Then, substituting \eqref{d-d} in \eqref{1182} yields the desired result.
\end{proof}

According to the monotonicity property of systemic performance measures, the inequality	\[ \tr \big( \triangledown \rhoo(L) \hat L\big) ~ \leq ~ 0  \]
holds for every Laplacian matrix $\hat L$.
This implies that  when weights of the candidate links are small enough, one can approximate the optimization problem \eqref{k-link} by the following optimization problem
\begin{equation}
\underset{\hat{\EE} \in \Pi_k(\EE_c)}{\textrm{minimize}}   \hspace{0.6cm} \tr \big( \triangledown \rhoo(L) \hat L\big), \label{k-link-B}
\end{equation}
where $\hat{L}$ is the Laplacian matrix of an appended candidate subgraph $\hat{\G} = (\mathcal V, \hat{\EE}, \varpi)$. Therefore, the problem boils down to select the $k$-largest elements of the following set    
\[\Big\{ \varpi(e) \big(\triangledown \rhoo(L)_{ii}+\triangledown \rhoo(L)_{jj}-\triangledown \rhoo(L)_{ij}-\triangledown \rhoo(L)_{ji}\big)\big|e=\{i,j\} \in \EE_c \Big\},\]
where $\varpi(e)$ is weight of link $e$. Table \ref{table-linear} presents our linearization approach as an algorithm. In some special cases,  one can obtain an explicit closed-form formula for systemic performance measure of the resulting augmented network. 

\begin{table}[t]
 \centering
 \caption{\small Linearization-based algorithm }
 {\small
\begin{tabular}{ |l| }  
\hline
\quad{\bf Algorithm:} Adding $k$ links using linearization                                                                                                                                                                                                                                                                                                                                                                                         \\ \hline \hline
 {\it Input:} $L$, $\mathcal E_c$, $\varpi$, and $k$ \\
1:  \quad {\it set} $\hat L=\mathbf 0$ \\
2: \quad {\it for} $i = 1$ {\it to} $k$\\
3:  \quad\quad \indent \indent  {\it find} $e=\{i,j\} \in \mathcal E_c$ that returns the maximum value for \\
 4: \quad\quad \indent \indent  ~~~~$\varpi(e) \big ( \triangledown \rhoo(L)_{ii}+\triangledown \rhoo(L)_{jj}-\triangledown \rhoo(L)_{ij}-\triangledown \rhoo(L)_{ji} \big)$ 
   \\5: \quad\quad \indent \indent  {\it set} the solution $e^{\star}$ 
  \\6: \quad\quad \indent \indent  {\it update} 
   \\7: \quad\quad\quad  \indent \indent
   $\hat L= \hat L + \varpi (e^\star)L_{e^{\star}}$, and
    \\8: \quad\quad\quad \indent \indent
    $\mathcal E_c = \mathcal E_c \, \backslash \,  \{e^\star \}$\\   9: \quad {\it end for}
   \\ \hline
   \end{tabular} \label{table-linear}}
\end{table}

\begin{theorem}
\label{col:h2-added}
Suppose that linear consensus network \eqref{first-order}-\eqref{first-order-G} with Laplacian matrix $L$ is endowed with systemic performance measure \eqref{zeta-measure} for $q=1$. Let us consider optimization problem \eqref{k-link}, where $\hat{L}$ is the Laplacian matrix of a candidate subgraph $\hat{\G} = (\mathcal V, \hat{\EE}, \varpi)$. Then, 
 \begin{equation*}
\zeta_1 (L+\epsilon \hat L) ~=~ \zeta_1 (L) - \epsilon \sum_{e \in \hat{\mathcal E}} \varpi(e)r_e(L^2)+\mathcal{O}(\epsilon^2), 
 \end{equation*}
where $r_e(L^2)$ is the effective resistance between the two ends of $e$ in a graph with node set $\mathcal V$ and Laplacian matrix $L^2$. 
\end{theorem}

\begin{proof}
We use the following identity 
	\begin{equation}
		\left({A} + \epsilon{X}\right)^{-1}
		= {A}^{-1}
		- \epsilon {A}^{-1} {X} {A}^{-1} + \mathcal{O}(\epsilon^2),
 		\label{w1}
	\end{equation}
for given matrices $A, X \in \R^{n \times n}$.
Based on \cite[Theorem 4]{Siami13cdc}, the performance measure $\zeta_1(.)$ can be calculated by
	\begin{equation}
 		\zeta_1 (L+\epsilon \hat{L}) ~=~ \tr((L+\epsilon \hat{L})^{\dag}). 
 		\label{w2}
	\end{equation}
Moreover, according to the definition of the Moore-Penrose generalized matrix inverse, we have 
	\begin{equation*}
		\left (L+\epsilon \hat L \right)^{\dag}~=~\left (\bar L+\epsilon \hat L\right)^{-1} - ~\frac{1}{n}J_n,
	\end{equation*}	
where $\bar L=L+\frac{1}{n}J_n$.
Using \eqref{w1} and \eqref{w2}, it follows that
	\begin{eqnarray}
		\left (L+\epsilon \hat L \right)^{\dag}~=~\bar L^{-1} - \frac{1}{n}J_n- \epsilon {\bar L}^{-1} {\hat L} {\bar L}^{-1} + \mathcal{O}(\epsilon^2).
 		\label{w3}
	\end{eqnarray}
Then we show that  
\begin{equation}
\tr ({\bar L}^{-1} {\hat L} {\bar L}^{-1} )~=~\tr ({\hat L} {\bar L}^{-2} )~=~\sum_{e \in \hat{\mathcal E}} \varpi(e)r_e(L^2).
\label{eq:624}
\end{equation}
Using \eqref{w2}, \eqref{w3} and \eqref{eq:624}, we get the desired result. 
\end{proof}

According to Theorem \ref{col:h2-added}, when weights of the candidate links are small, in order to solve problem \eqref{k-link}, it is enough to find $k$-largest element of the following set 
\[ \big\{ \varpi(e)r_e(L^2)~\big|~  e \in \mathcal E_c\big\}.\] 
Since the weights of the candidate links are given,  we only need to calculate the effective resistance $r_e(L^2)$ for all $e \in \mathcal E_c$.

As we discussed earlier, the design problem \eqref{k-link} is generally NP-hard. Our proposed approximation algorithm in this section works in polynomial-time. In example \ref{comparing}, we  discuss and compare optimality gap and time complexity of this method with other methods. The computational complexity of the linearization-based algorithm in Table \ref{table-linear} is $\mathcal O(n^3)$ for a given differentiable systemic performance measure from Table \ref{matrix-operator}. This involves computation of $\triangledown \rho$ for the original graph, which requires $\mathcal O(n^3)$ operations. The rest of the algorithm can be done in $\mathcal O(p k)$ for small $k$ and $\mathcal O(p \log p)$ operations for large $k$.

\section{Greedy Approximation Algorithms}\label{sec:algorithms}

In this section, we propose an optimal algorithm to solve the network growing problem \eqref{k-link} when $k=1$. It is shown that for some commonly used systemic performance measures, one can obtain a closed-form solution for $k=1$. We exploit our results and propose a simple greedy approximation algorithm for \eqref{k-link} with $k > 1$ by adding candidate links one at a time. For some specific subclasses of systemic performance measures, we prove that our proposed greedy approximation algorithm enjoys guaranteed performance bounds with respect to the optimal solution of the combinatorial problem \eqref{k-link}. Finally, we discuss time complexity of our proposed algorithms.

\begin{table}[t]
 \centering
    \caption{\small Simple greedy algorithm}
    {\small
\begin{tabular}{|l| }  
\hline
\quad{{\bf Algorithm: } Adding links Consecutively}                                                                                                                                                                                                                                                                                                                                                                                     \\ \hline \hline
 {\it Input:} $L$, $\mathcal E_c$, $\varpi$, and $k$ \\
1: {\it set} $\tilde L = L $\\
2: {\it for} $i = 1$ {\it to} $k$\\
3:  \quad \indent \indent  {\it find} link $e \in \mathcal E_c$ with maximum $\rhoo(\tilde L)-\rhoo(\tilde L+\varpi(e)L_e)$ \\
4:  \quad \indent \indent  {\it set} the solution $e^{\star}$ \\
5:  \quad \indent \indent  {\it update}  \\
6:  \quad\quad  \indent \indent
   $\tilde L= \tilde L + \varpi (e^\star)L_{e^{\star}}$, and\\ 
 7: \quad\quad \indent \indent $\mathcal E_c = \mathcal E_c \, \backslash \,  \{e^\star \}$\\
  8:  {\it end for}
   \\ \hline
   \end{tabular} }
   \label{greedy-table}
   \vspace{-0.4cm}
\end{table}

\subsection{Simple Greedy by Sequentially Adding Links}
\label{sec:291}
The problem of adding only one link can be formulated as follows
\begin{equation}
\underset{ e \in \EE_c}{\textrm{minimize}}   \hspace{0.6cm} \rhoo (L + L_e), \label{1-link}
\end{equation}
where $L_e$ is the Laplacian matrix of a candidate subgraph $\hat{\G}_e = (\mathcal V, \{e\}, \varpi)$. Let us denote the optimal cost of  \eqref{1-link} by $\mathrm{r}_1^*(\varpi)$. In order to formulate the optimal cost value of  \eqref{1-link}, we need to define the notion of a companion operator for a given systemic performance measure. 	
	
\begin{lemma}\label{psi-thm}
For a given systemic performance measure $\rhoo: \LL \rightarrow \R$, there exists a companion operator $\psii: \LL \rightarrow \R$ such that
	\begin{equation}
		\rhoo(L) = \psii(L^{\dag}),
		\label{eq:353}
	\end{equation}
for all $L \in \mathfrak L_n$. Moreover, the companion operator of $\rhoo$ is characterized by
\begin{equation}
  \psi(X) = \Phi(\mu_n^{-1}, \ldots, \mu_2^{-1}),
 \label{eq:453}
 \end{equation}
for all $X \in \mathfrak L_n$ with eigenvalues $\mu_2 \leq \ldots \leq \mu_n$, where operator $\Phi: \R^{n-1} \rightarrow \R$ is defined by \eqref{spectral-rho}. 
\end{lemma}

\begin{proof}
According to Theorem \ref{thm:schur-convex}, there exists a Schur-convex spectral function $\Phi: \R^{n-1} \rightarrow \R$ such that 
\[ \rhoo(L)~=~ \Phi(\lambda_2, \ldots, \lambda_n). \]
In addition, we know that for the Moore-Penrose pseudo-inverse of matrix $L \in \mathfrak L_n$, we have the following
\[ \lambda_i(L^\dag)~=~\lambda_{n-i+1}^{-1}(L) ~=~ \lambda_{n-i+1}^{-1},\]
for $i=2, \ldots n$, and $\lambda_1(L)=\lambda_1(L^{\dag})=0$. Consequently, we can rewrite $\rho(L)$ using its companion operator as
\[ \rho(L)=\Phi \left (\lambda_n^{-1}(L^{\dag}), \ldots, \lambda_2^{-1}(L^{\dag}) \right).\]
Therefore, by defining $\psii: \mathfrak L_n \rightarrow \R$ as \eqref{eq:453}, we get identity  \eqref{eq:353}.
\end{proof}

Table \ref{table-11} shows some important examples of systemic performance measure and their corresponding companion operators.
	
\begin{theorem}\label{lem-ex}
Suppose that a linear consensus network  \eqref{first-order}-\eqref{first-order-G} endowed by a systemic performance measure  $\rhoo: \LL \rightarrow \R$ is given.  The optimal  cost value of the optimization problem \eqref{1-link} is given by    
	\begin{equation}
	 \mathrm{r}_1^*(\varpi)~=~ \min_{e \in \EE_c} ~\psii \left(L ^{\dag}-  \frac{1}{\varpi^{-1}(e)+r_e(L)}U_e\right),  
	\label{eq:472}
	\end{equation}
where $\psii$ is the corresponding companion operator of $\rho$ and $U_e$ for a link $e=\{i,j\}$ is a rank-one matrix defined by
\begin{equation}
U_e~=~(L^{\dag}_i-L^{\dag}_j) (L^{\dag}_i-L^{\dag}_j)^{\text T}, \label{U-e}
\end{equation}
in which $L^{\dag}_i$ is the $i^{\text{th}}$ column of matrix $L^{\dag}$.
\end{theorem}

\begin{proof}
We use the following matrix identity 
\[(L+L_{e})^{\dag}~=~\left (\bar L+E_e\varpi(e)E^{\text T}_e \right )^{-1}-\frac{1}{n}J_n,\]
where $E_e$ is the incidence matrix  of graph $\hat{\G}_e$ and $\bar L=L+\frac{1}{n}J_n$. By utilizing the Woodbury matrix identity, we get
	\begin{equation}
		(L+L_{e})^{\dag} ~=~ L^{\dag}- \bar L^{-1}E_e\left(w^{-1}_1(e)+E^{\text T}_e\bar L^{-1}E_e\right)^{-1}E^{\text T}_e\bar L^{-1}.
		\label{eq:1325}
	\end{equation}
From the definition of the effective resistance between nodes $i$ and $j$, it follows that 
	\[r_e(L)~=~E^{\text T}_e\bar L^{-1}E_e~=~ l^{\dag}_{ii}+l^{\dag}_{jj}-l^{\dag}_{ij}-l^{\dag}_{ji}.\]
On the other hand, we have 
	\begin{equation}
 		\bar L^{-1}E_e~=~\left(L^{\dag}-\frac{1}{n}J_n\right)E_e~=~L^{\dag}E_e~=~L^{\dag}_i-L^{\dag}_j.
 		\label{eq:1333}
 	\end{equation}
Therefore, using \eqref{eq:1325} and \eqref{eq:1333}, we have
	\begin{eqnarray}
		&&\hspace{-.6 cm}(L+L_{e})^{\dag}=L^{\dag}- \frac{1}{\varpi^{-1}(e)+r_e(L)} (L^{\dag}_i-L^{\dag}_j)(L^{\dag}_i-L^{\dag}_j)^{\text T} \nonumber\\
		&&~~~~~~~= L ^{\dag}-  \frac{1}{\varpi^{-1}(e)+r_e(L)}U_e.
		\label{eq:522}
	\end{eqnarray}
From \eqref{eq:353} and \eqref{eq:522}, we can conclude the desired equation \eqref{eq:472}.	
\end{proof}

In some special cases, the optimal solution \eqref{eq:472} can be computed very efficiently using a simple separable update rule.   
	 
\begin{theorem}\label{coro:1241}
Suppose that linear consensus network \eqref{first-order}-\eqref{first-order-G} with Laplacian matrix $L$  is given. Then, for every link $e \in \mathcal{E}_c$ we have 
\begin{eqnarray*}
& &\hspace{-.8cm} \Scale[1.1]{\zeta_1(L+L_{e}) = \zeta_{1} (L) -  \frac{r_{e}(L^2)}{\varpi^{-1}({e})+r_{e}(L)},} \\
& &{ \hspace{-.8cm} \Scale[1.1]{\zeta_{2}^2 (L+L_{e}) = \zeta_{2}^2(L) + \left[\frac{r_e(L^2)}{\varpi^{-1}(e)+r_e(L)}\right]^2  - \frac{ 2 r_e(L^3)}{\varpi^{-1}(e)+r_e(L)}}},\\
& & \hspace{-.8cm} \Scale[1]{\upsilon(L+L_{e}) ~=~ \upsilon (L) ~-~ \log \big(1+ r_{e}(L)\varpi(e)\big)},
\end{eqnarray*}
where $r_{e}(L^m)$ is the effective resistance between the two ends of link $e$ in a graph with node set $\mathcal{V}$ and Laplacian matrix $L^m$ for $m \in \{1,2,3\}$.
\end{theorem}

\begin{proof}
Based on Theorem \ref{lem-ex}, it is straightforward to get the desired result for $\zeta_1(.)$ and $\zeta_2(.)$. For the last part, using the definition of $\upsilon(.)$ and \eqref{measure:uncertainty}, we get
\begin{eqnarray}
	\upsilon(L+L_{e}) &=& \log \det \left(\frac{1}{2}(L+L_e)^\dag+\frac{1}{n}J_n\right)\nonumber \\
	&=&  \log \det \left( 2(L+L_e)+\frac{1}{n}J_n\right)^{-1}.
	\label{eq:2303}
\end{eqnarray}
According to the matrix determinant lemma we have
\begin{equation}
\det( A+uv^\mathrm{T}) ~=~ (1 + {v}^\mathrm{T}{A}^{-1}{u})\,\det({A}).
\label{eq:1520}
\end{equation}
Now using \eqref{eq:2303} and \eqref{eq:1520}, it follows that
\begin{eqnarray*}
\det \left( 2(L+L_e)+\frac{1}{n}J_n\right)^{-1} & = & \\
 & & \hspace{-2.5cm} \det \left (\left (2L+\frac{1}{n}J_n\right) ^{-1}-  \frac{1}{2\varpi^{-1}(e)+2r_e(L)}U \right)\\
&&\hspace{-2.5cm}  = \left(1 -  \frac{r_e(L)}{\varpi^{-1}(e)+r_e(L)}\right)\det \left (2L+\frac{1}{n}J_n\right) ^{-1}, 
\end{eqnarray*}
then by taking $\log$ from both sides, we get the desired result.
\end{proof}
 
In these special cases, the computational complexity of calculating the optimal solution for network design problem \eqref{1-link} is relatively low. For $q=1$, the optimal cost value is equal to $\zeta_1(L+L_{e^*})$, where 
\begin{equation}
e^* = \arg \max_{e \in \EE_c}   ~  \frac{r_e(L^2)}{\varpi^{-1}(e)+r_e(L)}, \label{1-link-B}
\end{equation}
and for $q=2$, the optimal cost value is equal to $\zeta_{2} (L+L_{e^*})$, where 
\begin{equation*}
e^* = \arg \min_{e \in \EE_c}   ~  \left (\left[\frac{r_e(L^2)}{\varpi^{-1}(e)+r_e(L)}\right]^2  - \frac{ 2 r_e(L^3)}{\varpi^{-1}(e)+r_e(L)} \right) \end{equation*}
Moreover, for \eqref{measure:uncertainty}, the optimal cost value is equal to $\upsilon(L+L_{e^*})$, where 
\begin{equation*}
e^* = \arg \min_{e \in \EE_c}   ~   \log \big(1+ r_{e}(L)\varpi(e)\big). 
\end{equation*}
The location of the optimal link is sensitive to its weight. For example when optimizing with respect to $\zeta_1$, 
maximizers of $r_e(L)$, $r_e(L^2)$ and $r_e(L^2)/r_e(L)$ can be three different links. In Example \ref{ex:2} and Fig. \ref{fig:1618} of Section \ref{sec:simu}, we illustrate this point by means of a simulation.  Furthermore, one can obtain the following useful fundamental limits on the best achievable cost values. 
\begin{theorem}\label{funda-limit}
Let us denote the value of performance improvement by adding an edge $e$ with an arbitrary positive weight to linear consensus network  \eqref{first-order}-\eqref{first-order-G} by 
\[ \Delta \rho(L) = \rho(L) - \rho(L+L_{e}).\]  Then, the maximum achievable performance improvement is
\begin{equation}
\Delta \rho(L) ~\leq~\psi(L ^{\dag}) - \psi \Big(L ^{\dag}-  r_e(L)^{-1}U_e\Big), \label{max-rho}
\end{equation}
where $U_e$ is given by \eqref{U-e} and the upper bound can be achieved as $w$ tends to infinity. Moreover, we have the following explicit fundamental limits 
\begin{eqnarray} 
\Delta \zeta_1(L) & \leq &    \frac{r_e(L^2)}{r_e(L)}, \\
\Delta \zeta_2^2(L) & \leq &  \left[\frac{r_e(L^2)}{r_e(L)}\right]^2  - 2\frac{r_e(L^3)}{r_e(L)}.
\end{eqnarray}
\end{theorem}

 \begin{table*}[t]
{ \small
\begin{center}
    \begin{tabular}{ | p{5.29cm}  | l | p{5.0cm}  | p{5.7cm} |}
    \hline
 Systemic   Performance Measure & Symbol & Spectral Representation &  The Corresponding Companion Operator \\ \hline \hline
    Spectral zeta function & ${\zeta}_{q}(L)$ & $\displaystyle \Big( \sum_{i=2}^n \lambda_i^{-q} \Big)^{1/q}$ & $\displaystyle \Big( \sum_{i=2}^n \mu_i^{q} \Big)^{1/q}$ for $q \geq 1$ \\
    \hline
        Gamma entropy & $I_{\gamma}(L)$ & $\displaystyle  \gamma^2 \sum_{i=2}^n  \Big(\lambda_i- \big(\lambda_i^2-\gamma^{-2}\big)^{\frac{1}{2}} \Big)$ &  $\displaystyle  \gamma^2 \sum_{i=2}^n  \Big(\mu_i^{-1}- \big(\mu_i^{-2}-\gamma^{-2}\big)^{\frac{1}{2}} \Big)$
        \\
    \hline
    Expected  transient output covariance & $\tau_t(L)$ & $\displaystyle  \frac{1}{2} \sum_{i=2}^n \lambda_i^{-1} (1- e^{-\lambda_i t})$ & $\displaystyle \frac{1}{2} \sum_{i=2}^n \mu_i (1- e^{-\frac{t}{\mu_i}})$
        \\
    \hline    
System Hankel norm        & $\eta(L)$ & $\displaystyle  \frac{1}{2}\lambda_2^{-1}
$ & $\frac{1}{2} \mu_n$ \\
    \hline  
 Uncertainty volume of the output          & $\upsilon(L)$ & $\displaystyle  (1-n) \log 2 - \sum_{i=2}^n \log \lambda_i $ & $\displaystyle  (1-n) \log 2 + \sum_{i=2}^n \log \mu_i $
        \\
    \hline  
Hardy-Schatten system norm or $\mathcal{H}_p$-norm             & $\theta_p(L)$ & $\displaystyle \left\{ \frac{1}{2\pi} \int_{-\infty}^{\infty} \sum_{k=1}^n \sigma_k(G(j \omega))^p  \hspace{0.05cm} d\omega \right\}^{1/p}$  $ = \alpha_0 \left( \tr \left( L^\dag \right)^{p-1}\right)^{\frac{1}{p}}$ &  $\alpha_0 \displaystyle \bigg( \sum_{i=2}^n \mu_i^{p-1} \bigg)^{1/p}$ for $2 \leq p \leq \infty$, where $\alpha_0^{-1}=\sqrt[p]{-\beta(\frac{p}{2},-\frac{1}{2})}$.
        \\
    \hline   
    \end{tabular}
        \caption{ \small Some important examples of spectral systemic performance measures and their corresponding companion operators.} \label{table-11}
\end{center}
   \vspace{-0.5cm}}
\end{table*}

\begin{proof}
We utilize monotonicity property of  companion operator of a systemic performance measure, i.e., If $L_1^{\dag} \preceq  L_2^{\dag} $, then
\[ \psi(L_1^{\dag}) ~\leq~  \psi(L_2^{\dag}), \]
and the inequality 
\[ L ^{\dag}-  r_e(L)^{-1}U_e ~\preceq~ L ^{\dag}-  \frac{1}{w^{-1}+r_e(L)}U_e \] 
to show that 
\[\psi \Big(L ^{\dag}-  r_e(L)^{-1}U_e\Big) ~\leq~ \psi\Big(L ^{\dag}-  \frac{1}{w^{-1}+r_e(L)}U_e\Big). \]
From this inequality, we can directly conclude \eqref{max-rho}. For systemic performance measure $\zeta_1(.)$, inequality \eqref{max-rho} reduces to 
\begin{eqnarray}  
\Delta \zeta_1(L) &\leq& \tr(L ^{\dag}) - \tr \Big(L ^{\dag}-  r_e(L)^{-1}U_e\Big), \nonumber \\
&=&\tr \left ( r_e(L)^{-1}U_e\right) ~=~  r_e(L)^{-1} \tr \left (U_e\right).
\label{eq:1869}
\end{eqnarray}
Moreover, based on the definition of $U_e$, we have
\[ \tr (U_e) ~=~ \tr (L^{\dag}E_e^{\rm T}  E_e L^{\dag})~=~E_e L^{\dag, 2} E_e^{\rm T}~=~r_e(L^2). \]
Using this and \eqref{eq:1869}, it follows that
\[\Delta \zeta_1(L) ~\leq~\frac{r_e(L^2)}{r_e(L)}.\]
Similarly for $\zeta_2^2(.)$, using  \eqref{max-rho} and the definition of $\zeta_2(.)$, results in
\begin{eqnarray}  
\Delta \zeta_2^2(L) &\leq& \tr \left (L ^{\dag,2}\right ) - \tr \left(\Big(L ^{\dag}-  r_e(L)^{-1}U_e\Big)^2\right), \nonumber \\
&=&\frac{1}{r_e^2(L)} \tr \left (U_e^2\right) - 2 \tr \left( r_e(L)^{-1}U_e L^{\dag}\right) \nonumber \\
&=&\left [ \frac{r_e(L^2)}{r_e(L)}\right ]^2 -  2 \frac{ \tr \left( U_e L^{\dag}\right)}{r_e(L)}  \nonumber \\
&=&\left [ \frac{r_e(L^2)}{r_e(L)}\right ]^2 -  2 \frac{ \tr \left( L^{\dag}E_e^{\rm T}  E_e L^{\dag, 2} \right) }{r_e(L)}  \nonumber \\
&=&\left [ \frac{r_e(L^2)}{r_e(L)}\right ]^2 -  2 \frac{ r_e(L^3) }{r_e(L)}.
\end{eqnarray}
This completes proof.
\end{proof}

The result of Theorem \ref{funda-limit} asserts that, in general,  performance improvement may not be arbitrarily large by adding only one new link. In some cases, however, performance improvement can be arbitrarily good. For instance, for the uncertainty volume of the output, we have
\begin{equation}
\lim_{\varpi(e) \rightarrow +\infty} ~\Delta \upsilon(L) = +\infty.
\end{equation}

The result of Theorem \ref{lem-ex} can be utilized to devise a greedy approximation method by decomposing \eqref{k-link} into $k$ successive tractable problems in the form of \eqref{1-link}. In each iteration, Laplacian matrix of the network is updated and then optimization problem \eqref{1-link} finds the next best candidate link as well as its location.  Since the value of systemic performance measure can be calculated explicitly in each step using Theorem \ref{lem-ex}, one can explicitly calculate the value of systemic performance measure for the resulting augmented network. This value can be used to determine the effectiveness of this method. Table \ref{greedy-table} summarizes all steps of our proposed greedy algorithm, where the output of the algorithm is the Laplacian matrix of the resulting augmented network. In Section \ref{sec:simu}, we present several supporting numerical examples.

\begin{remark}
The optimization problem \eqref{1-link} with performance measure $\zeta_{\infty}(L)=\lambda^{-1}_2$ was previously considered in \cite{Ghosh2006}, where a heuristic algorithm was proposed to compute an approximate solution. Later on, another  approximate method for this problem was presented in \cite{Kolla}. Also, there is a similar version of this problem that is reported in \cite{Fardad}, where the author studies convergence rate of circulant  consensus networks by adding some long-range links. Moreover, a non-combinatorial and relaxed version of our problem of interest has some connections to the sparse consensus network design problem \cite{mogjovACC15, wujovACC14, farlinjovTAC14sync}, where they consider $\ell_1$-regularized $\mathcal H_2$ optimal control problems. When the candidate set $\mathcal E_c$ is the set of all possible links except the network links, i.e., ~$\mathcal E_c= \mathcal V \times \mathcal V \, \backslash \,  \mathcal E$, and the performance measure is the logarithm of the uncertainty volume, our result reduces to the result reported in \cite{Summers16}. 
\end{remark}

\subsection{Supermodularity and Guaranteed Performance Bounds}\label{subsec1}

A systemic performance measure is a continuous function of link weights on the space of Laplacian matrices $\LL$. Moreover, we can represent a systemic performance measure equivalently as a set function over the set of weighted links. Let us denote by $\GG(\V)$ the set of all weighted graphs with a common node set $\V$. 

\begin{definition}
\label{defin:3} For  a given systemic performance measure $\rhoo: \LL \rightarrow \R$, we associate a set function $\tilde \rho: \GG(\V) \rightarrow \R$ that is defined as
\[ \tilde \rhoo(\G)~=~ \rhoo \bigg(\sum_{e \in \EE} w(e)L_e \bigg)~=~\rhoo (L),\]
where $L$ is Laplacian matrix of $\G=(\V, \EE, w)$ and $L_e$ is the Laplacian matrix of $(\V, \{e\}, 1)$, which is an unweighted graph formed by a single link $e$.
\end{definition}

\begin{definition}  The union of two  weighted graphs $\G_1=(\V, \EE_1, w_1)$ and $\G_2=(\V, \EE_2, w_2)$ is defined   as follows
\[\G_1 \vee \G_2~ := ~ (\V,  \EE_1 \cup \EE_2, w)\] 
in which 
\begin{eqnarray}
w(e):=\begin{cases}\max\{w_1(e),w_2(e)\} ~~\text{if}~e \in \EE_1 \cup \EE_2 \\
0~~~~~~~~~~~~~~~~~~~~~~~~~\text{otherwise}
 \end{cases}.
\end{eqnarray}
\end{definition}

\begin{definition} The intersection of  two  weighted graphs $\G_1=(\V, \EE_1, w_1)$ and $\G_2=(\V, \EE_2, w_2)$ is defined   as follows
\[\G_1 \wedge \G_2~ := ~ (\V,  \EE_1 \cap \EE_2, w)\] 
in which 
\begin{eqnarray*}
w(e):=\begin{cases}\min\{w_1(e),w_2(e)\} ~~\text{if}~e \in \EE_1 \cap \EE_2 \\
0~~~~~~~~~~~~~~~~~~~~~~~~~\text{otherwise}
 \end{cases}.
\end{eqnarray*}
\end{definition}


The following definition is adapted from  \cite{combinatorial} for our graph theoretic setting.

\begin{definition}
A set function $\tilde \rho: \GG(\V) \rightarrow \R$ is supermodular with respect to the link set if it satisfies 
\begin{equation}
 \tilde \rhoo(\G_1 \wedge \G_2) \, + \, \tilde \rhoo(\G_1 \vee \G_2) \, \geq \,  \tilde \rhoo(\G_1) \, + \, \tilde \rhoo(\G_2)
\label{eq:352}
 \end{equation}
\end{definition}

 \begin{figure}[t]
	\centering
	\begin{tabular}{c c c}
	\includegraphics[trim = 50 10 30 10, clip,width=.14 \textwidth]{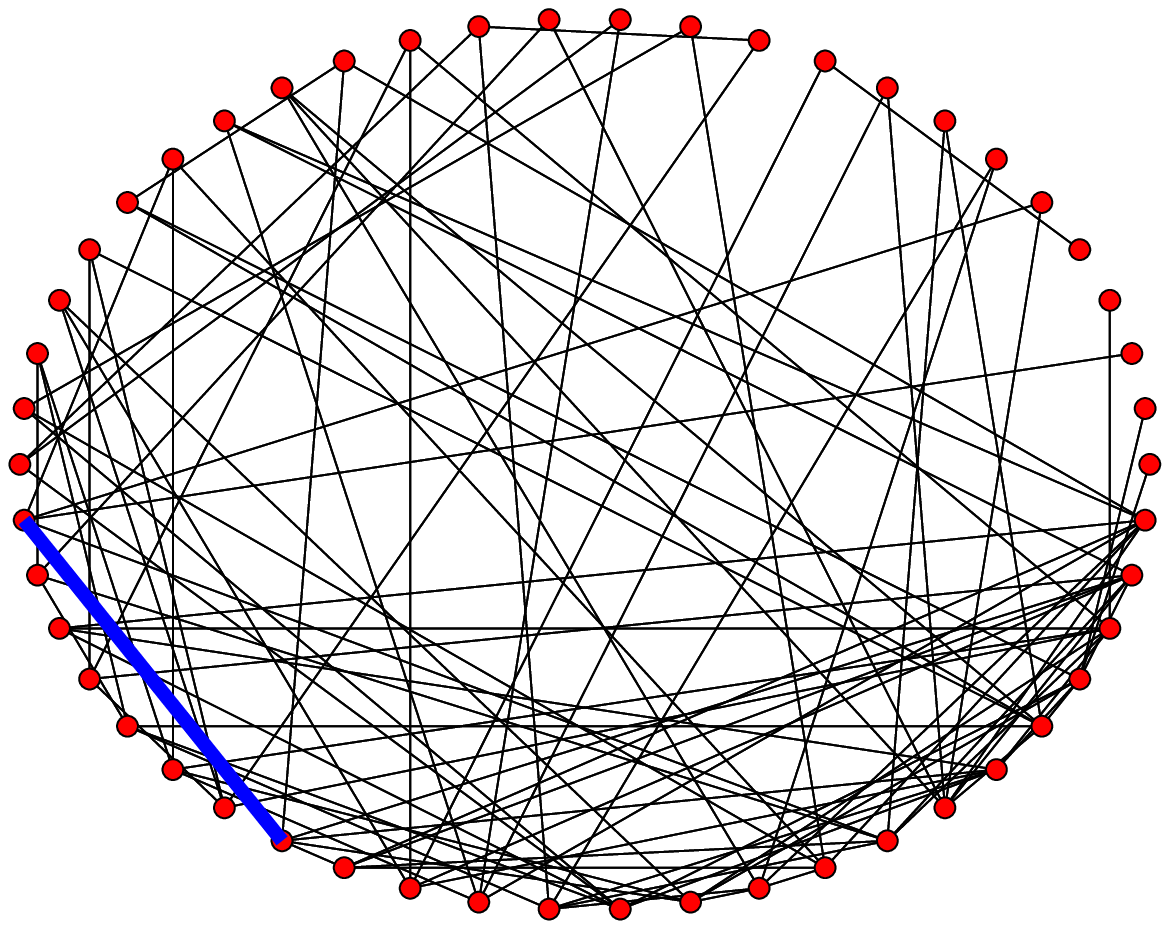}& 
	\includegraphics[trim = 50 10 30 10, clip,width=.14 \textwidth]{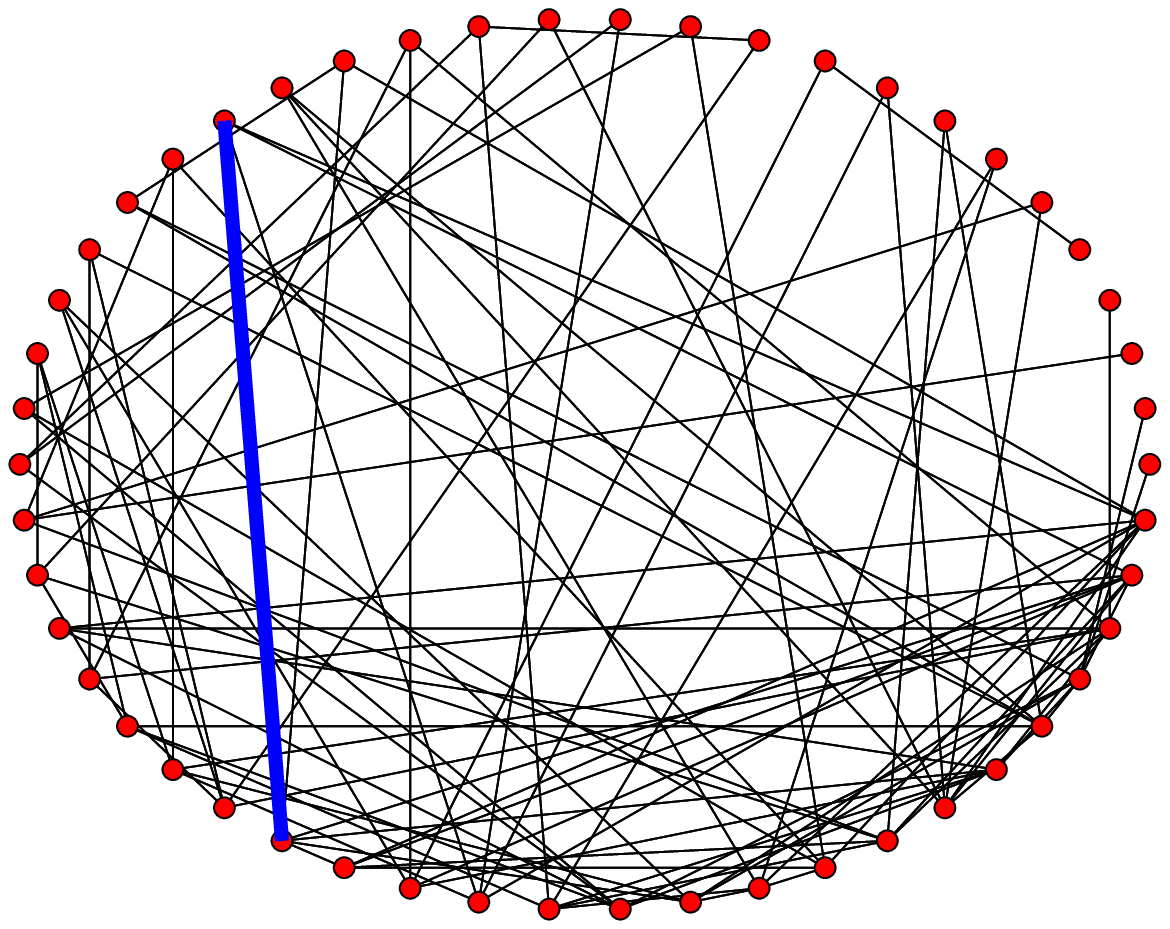} & 
	\includegraphics[trim = 50 10 30 10, clip,width=.14 \textwidth]{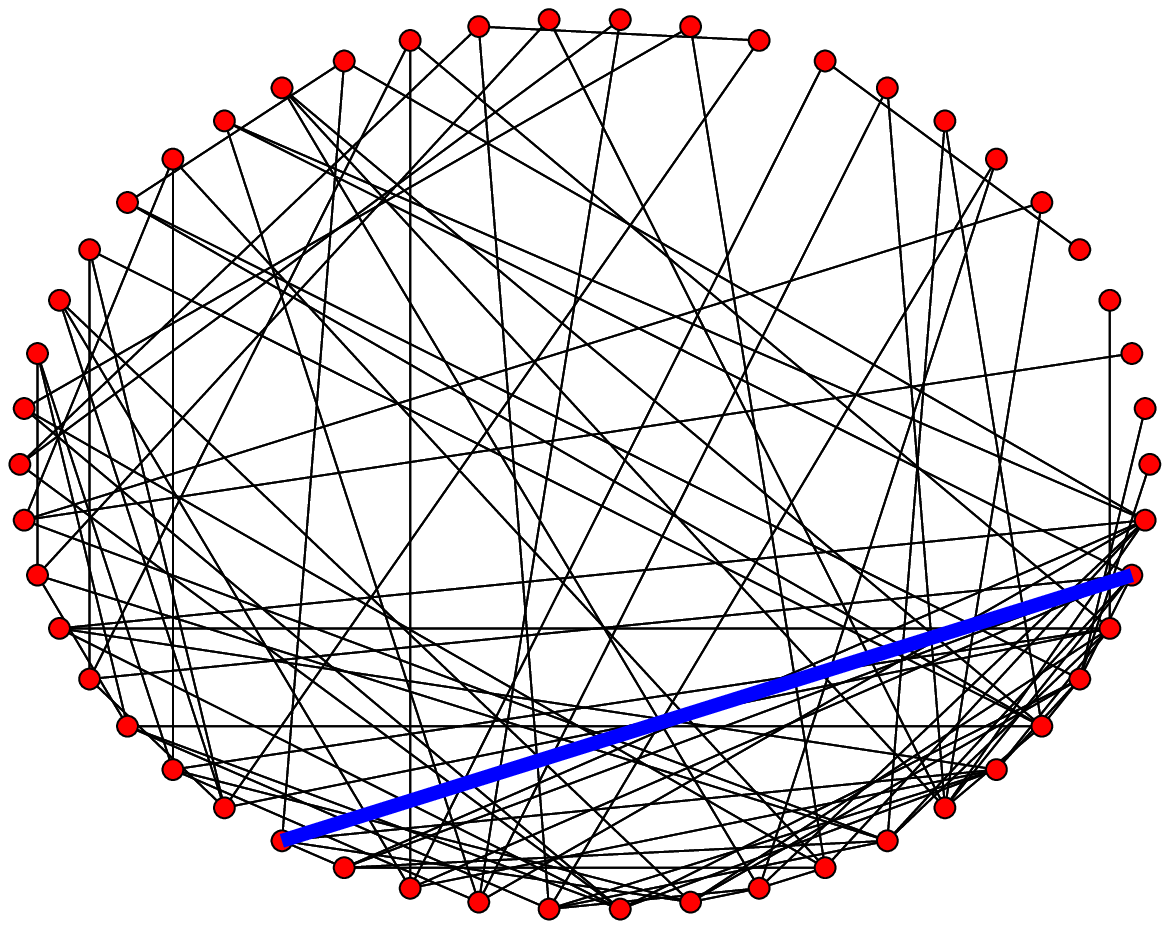}\\ (a) & (b) & (c)
	\end{tabular}
	\caption{\small The interconnection topology of all three graphs, except for their highlighted blue links, are identical, which show the coupling graph of the linear consensus network in Example \ref{ex:2}. The coupling graph shown in here is a generic connected graph with $50$ nodes and $100$ links, which are drawn by black lines. The  optimal links are shown by blue line segments.  }
  	\label{fig:1618}
	\vspace{-0.8cm}
\end{figure}

\begin{theorem}
\label{submodular:th}
Suppose that systemic performance measure $\rho: \LL \rightarrow \R$ is differentiable and $\triangledown \rhoo: \LL \rightarrow \R^{n \times n}$ is monotonically increasing with respect to the cone of positive semidefinite matrices\footnote{$L_1 \preceq L_2 \Longrightarrow \triangledown \rhoo(L_1) \preceq \triangledown \rhoo(L_2)$.}. Then, the corresponding set function $\tilde \rho: \GG(\V) \rightarrow \R$, from Definition \ref{defin:3}, is supermodular.
\end{theorem}
\begin{proof}
We know that
\begin{equation} 
\frac{d}{dt}\rhoo(L+tX)~=~\tr(\triangledown \rhoo(L+tX) X).
 \label{eq:409}
\end{equation}
where $t \in \R_+$ and $L, X \in \LL$.
From \eqref{eq:409}, we get 
\begin{eqnarray}
&& \hspace{-.2cm}\frac{d}{dt}\big( \rhoo(L_1+tX) - \rhoo(L_2+tX) \big) \, = \, \nonumber \\
&&~~~~~~~~~~~~~~ \tr\left (\big(\triangledown \rhoo(L_1+tX) - \triangledown \rhoo(L_2+tX) \big)X\right),
 \label{eq:415}
 \end{eqnarray}
where $ L_1 ,L_2 \in \LL$ and $L_1 \preceq L_2$. From the monotonicity property of $\triangledown \rhoo$ and \eqref{eq:415}, we get 
\begin{equation}
 \frac{d}{dt}\big( \rhoo(L_1+tX) - \rhoo(L_2+tX) \big) \, \leq \, 0.
 \label{eq:474}
\end{equation}
Then, by taking integral from both sides of \eqref{eq:415}, and then using \eqref{eq:474} we have 
\begin{eqnarray*}
 &&\hspace{-.6cm}\int_0^1 \frac{d}{dt}\rhoo(L_1+tX) dt - \int_0^1 \frac{d}{dt}\rhoo(L_2+tX) dt ~\leq~ 0,
 \end{eqnarray*}
which directly implies that
\begin{equation}
 \rhoo(L_1+X)-\rhoo(L_1)  ~\leq~ \rhoo(L_2+X)-\rhoo(L_2).
\label{eq:428}
\end{equation}

On the other hand, the corresponding Laplacian matrices of  $\mathcal G_1$, $\mathcal G_2$, $\mathcal G_1 \wedge \mathcal G_2$, and $\mathcal G_1 \vee \mathcal G_2$ are given as follows  
\begin{eqnarray}
\begin{cases}
L_{\mathcal G_1}:=\sum_{e \in \mathcal E_1 } w_1(e)L_e,  \\

L_{\mathcal G_2}:=\sum_{e \in \mathcal E_2 } w_2(e)L_e, \\

L_{\mathcal G_1 \wedge \mathcal G_2}:=\sum_{e \in \mathcal E_1 \cap \mathcal E_2} \min \{w_1(e), w_2(e)\} L_e,\\

L_{\mathcal G_1 \vee \mathcal  G_2}:=\sum_{e \in \mathcal E_1 \cup \mathcal E_2} \max \{w_1(e), w_2(e)\} L_e.
\end{cases}
\label{eq:4322}
\end{eqnarray}
Based on these definitions, we have 
\begin{equation}
L_{\mathcal G_1 \wedge \mathcal G_2} ~\preceq~ L_{\mathcal G_1}, L_{\mathcal G_2} ~\preceq~ L_{\mathcal G_1 \vee \mathcal G_2}.
\label{eq:360}
\end{equation}
By setting $L_1=L_{\mathcal G_1 \wedge \mathcal G_2}$, $L_2=L_{\mathcal G_1} $, and $X = L_{\mathcal G_2}- L_{\mathcal G_1 \vee \mathcal G_2}$ in inequality \eqref{eq:428}, we get
\begin{eqnarray}
&&\hspace{-.5cm} \rhoo(L_{\mathcal G_1 \wedge \mathcal G_2}+L_{\mathcal G_2}-L_{\mathcal G_1\wedge \mathcal G_2 })-\rhoo(L_{\mathcal G_1 \wedge \mathcal G_2}) = \rhoo(L_{\mathcal G_2})-\rhoo(L_{\mathcal G_1 \wedge \mathcal G_2})  \nonumber \\
&&~~~~~~~~~~~\leq~ \rhoo(L_{\mathcal G_1 \vee \mathcal G_2}+L_{\mathcal G_2}-L_{\mathcal G_1\wedge \mathcal G_2 })-\rhoo(L_{\mathcal G_1 \vee \mathcal G_2}).
\label{eq:448}
\end{eqnarray}
According to \eqref{eq:4322}, we have
\begin{equation}
L_{\mathcal G_1 \vee \mathcal G_2}+L_{\mathcal G_1 \vee \mathcal G_2 }~=~ L_{\mathcal G_1}+L_{\mathcal G_2}.
\label{eq:452}
\end{equation}
Therefore, based on equality \eqref{eq:452} we can rewrite the right hand side of inequality \eqref{eq:448}, as follows
\begin{equation}
\rhoo(L_{\mathcal G_1 \vee \mathcal G_2}+L_{\mathcal G_2}-L_{\mathcal G_1 \wedge \mathcal G_2 })-\rhoo(L_{\mathcal G_1 \vee \mathcal G_2})= \rhoo(L_{\mathcal G_1})-\rhoo(L_{\mathcal G_1 \vee \mathcal G_2}).
\label{eq:458}
\end{equation}
Finally, using Definition \ref{defin:3}, \eqref{eq:448} and \eqref{eq:458}, we can conclude  \eqref{eq:352}. 
\end{proof}
}

It should be emphasized that convexity property of a systemic performance measure $\rho$ implies that $\triangledown \rho$, if it exists, is a monotone mapping\footnote{$\tr \left (  (\triangledown \rhoo(L_1) - \triangledown \rhoo(L_2))(L_1 -L_2) \right) \geq 0$, where $ L_1, L_2 \in \LL$.}. However, this property is not sufficient for supermodularity of its corresponding set function $\tilde \rho$. 

\begin{example}
In our first example, we show that the uncertainty volume of the output \eqref{measure:uncertainty} satisfies conditions of Theorem \ref{submodular:th}. The gradient operator of this systemic performance measure is 
\[ \triangledown \upsilon(L) = -\Big(L + \frac{1}{n} J_n \Big)^{-1}.\] 
It is straightforward to verify that $\triangledown \upsilon(L)$ is monotone with respect to the cone of positive semidefinite matrices. Thus, $\upsilon(L)$ is supermodular. 
\end{example}

\begin{figure}[t]
\centering
  	\begin{tabular}{c c}
    ~~~\includegraphics[trim = 45 10 35 10, clip,width=0.17\textwidth]{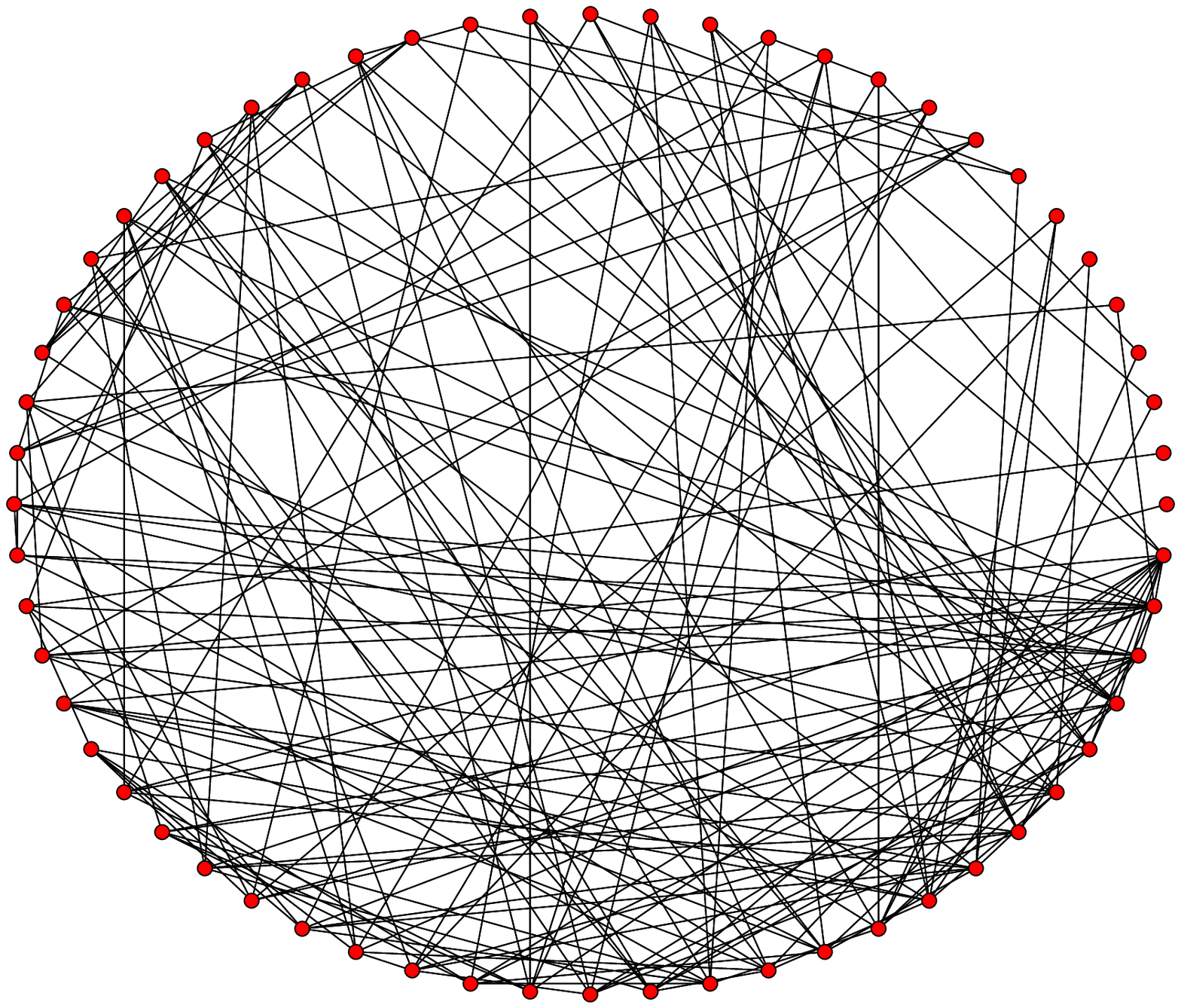}~~~~ &~~~
    \includegraphics[trim = 45 10 35 10, clip,width=0.17\textwidth]{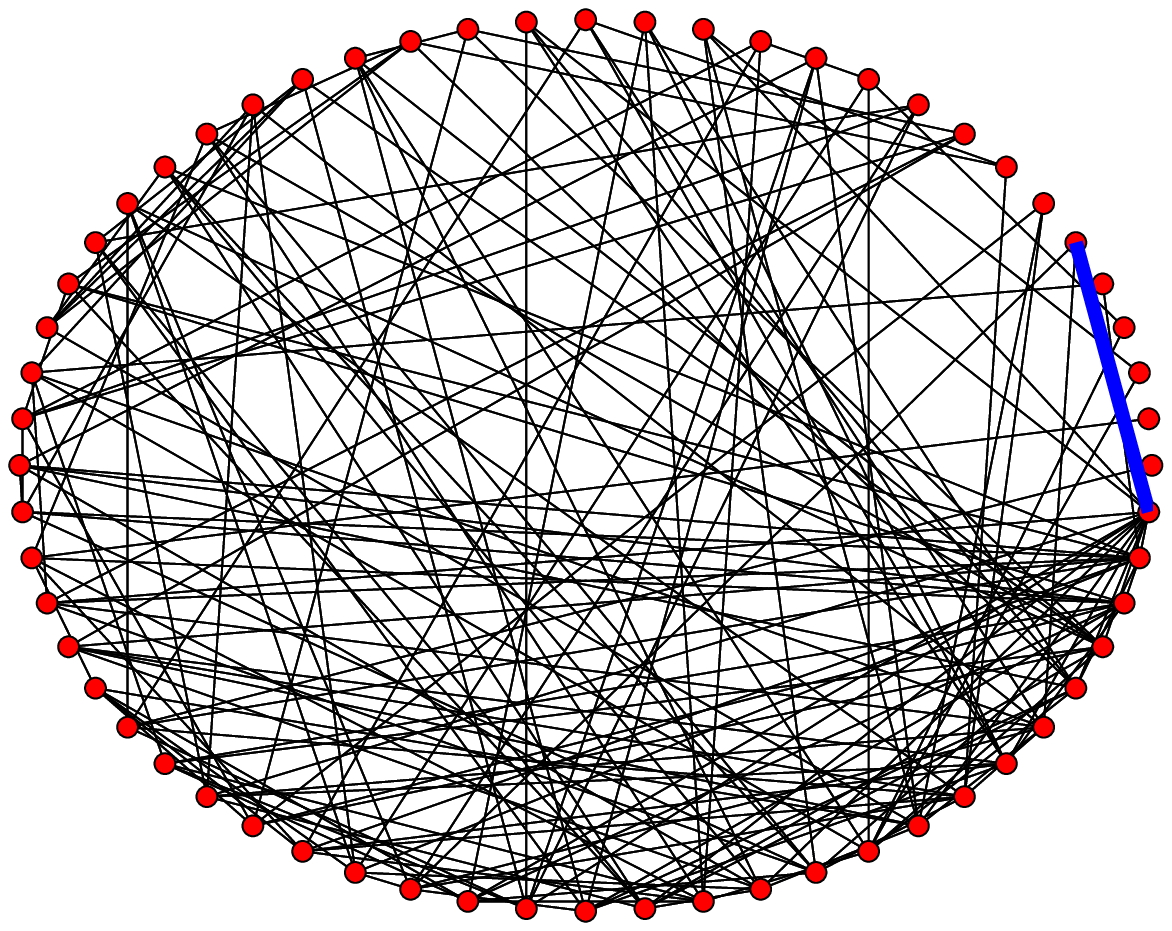}~~~\\ (a) &~ (b)
    \end{tabular}
   \caption{\small{ The coupling graph of the network used in Example \ref{ex:1} is shown in (a) that consists of $60$ nodes and $176$ links. The location of the optimal link, highlighted by the blue color, is shown in (b). }}
   \label{Fig:751}
   \vspace{-0.6cm}
\end{figure}

\begin{example}
In our second example, we consider a new class of systemic performance measures that are defined as 
\begin{equation}  
\mathfrak{m}_q(L) ~=~ - \sum_{i=2}^n \lambda_i^{q},
\label{z-q}
\end{equation}
where $0\leq q \leq 1$.
According to Theorem \ref{f-sum}, this spectral function is a systemic performance measure as function $-\lambda^q$ for $0\leq q \leq 1$ is a decreasing convex function on $\R_+$. Moreover, its gradient operator, which is given by $\triangledown \mathfrak{m}_q(L) =~ q L^{q-1}$ is monotonically increasing for all $0\leq q \leq 1$. Therefore, according to Theorem \ref{submodular:th}, systemic performance measure \eqref{z-q} is supermodular over the set of all weighted graphs with a common node set.
\end{example}

\begin{remark}
For a given performance measure $\rho$, there are several different ways to define an extended set function for $\rho$. These set functions may have different properties. For instance, the extended set function of $\zeta_1$ is supermodular over principle sub-matrices \cite{Submodular}, but it is not supermodular over the set of all weighted graphs with a common node set (see Definition \ref{defin:3}). 
\end{remark}

For those systemic performance measures that satisfy conditions of Theorem \ref{submodular:th}, one can provide guaranteed  performance bounds for our proposed greedy algorithm in Subsection \ref{sec:291}. The following result is based a well-known result from \cite[Chapter III, Section 3]{combinatorial}. 

\begin{theorem}
Suppose that systemic performance measure $\rhoo: \LL \rightarrow \R$ is differentiable and $\triangledown \rhoo: \LL \rightarrow \R^{n \times n}$ is monotonically increasing with respect to the cone of positive semidefinite matrices. Then, the greedy algorithm in Table \ref{greedy-table}, which starts with $\hat{\mathcal E}$ as the empty set and at every step selects an element $e \in \mathcal E_c$ that minimizes the marginal cost $\rhoo(L+L_{\hat{ \mathcal E}}+L_e)-\rhoo(L+L_{\hat{\mathcal E}})$, provides a set $\hat{\mathcal E}$ that achieves a $(1- 1/e)$-approximation\footnote{~ This means that  $\frac{\rhoo(L+\tilde L)- \rhoo(L)}{\rhoo(L+L^*) - \rhoo(L)}  \, \geq \, 1 -\frac{1}{e}$, where $L^*$ is the optimum solution and $\tilde L$ is the solution of the greedy algorithm, or equivalently: $ \frac{\rhoo(L+\tilde L)- \rhoo(L+L^*)}{\rhoo(L) - \rhoo(L+L^*) }  \leq \frac{1}{e}$,  where $e$ is Euler's number. } of the optimal solution of the combinatorial network synthesis  problem \eqref{k-link}. 
\end{theorem} 

Since the class of supermodular systemic performance measures are monotone, the combinatorial network synthesis  problem \eqref{k-link} is polynomial-time solvable with provable optimality bounds \cite{combinatorial}.  Supermodularity is not a ubiquitous property for all systemic performance measures. Nevertheless, our simulation results in Section \ref{sec:simu} assert that the proposed greedy algorithm in  Table \ref{greedy-table} is quite powerful and provides tight and near-optimal solutions for a broad range of systemic performance measures.

\begin{figure}[t]
  \begin{center} 
       \psfrag{X}[c][c]{\footnotesize Label of a candidate link}        
       \psfrag{Y}[c][c]{\footnotesize $\zeta_1$}     
       \psfrag{B}[c][c]{\footnotesize ~~~~~~~~~~~~the optimal value }   
    \includegraphics[width=0.35\textwidth]{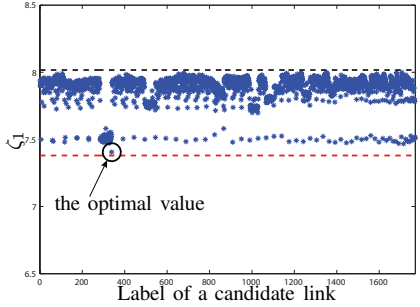}
   \end{center}
   \caption{\small{This plot is discussed in  Example \ref{ex:1}. }}
   \label{Fig:740}
   \vspace{-0.5cm}
\end{figure}

\subsection{Computational Complexity Discussion}

As we discussed earlier, the network synthesis problem \eqref{k-link} is in general NP-hard. However, this problem is solvable when $k=1$ and the best link can be found by running an exhaustive search over all possible scenarios, i.e., by  calculating the value of a performance measure for all possible $p$ augmented networks, where $p$ is the number of candidate links. The computational complexity of evaluating  performance of a given linear consensus network depends on the specific choose of a systemic  performance measure.  Let us denote  computational complexity of a given systemic performance measure $\rho: \LL \rightarrow \R$ by $\mathcal O\left(M_\rho(n)\right)$. 
In the simple greedy algorithm of Table \ref{greedy-table}, the difference term
\begin{equation}
\rhoo(\tilde L)-\rhoo \big(\tilde L+\varpi(e)L_e \big)
\label{dif}
\end{equation}
is calculated and updated for each candidate link at each step, for the total of $k \big(p -\frac{k-1}{2}\big)$ times. Thus, the total computational complexity of our simple greedy algorithm is $\mathcal O\left(M_{\rho}(n) (p -\frac{k-1}{2})k\right)$ operations. This computational complexity is at most $\mathcal O \left ( M_{\rho}(n) n^2 k\right)$, where $p=\binom{n}{2}$, i.e., when the candidate set contains all possible links. The complexity of the brute-force method is $\mathcal O\left (M_{\rho}(n) \binom{p}{k}\right)$\footnote{~This corresponds to calculating the value of a performance measure for all $\binom{p}{k}$ possible augmented  networks.}. This can be at most $\mathcal O \left( M_\rho(n) 2^p/\sqrt p\right)$. 
Moreover, if $k \leq \sqrt p$, then the computational complexity will be $\mathcal O\left (M_{\rho}(n) p^k/k! \right)$. 

In some occasions, we can take advantage of the rank-one updates in Theorems \ref{lem-ex} and \ref{coro:1241}, where it is shown that a rank-one deviation in a matrix results in a rank-one change in its inverse matrix as well. This helps reduce the 
computational complexity of \eqref{dif} to the order of $\mathcal  O(n^2)$ instead of $\mathcal O(n^3)$ operations. As it is shown in \cite{Yaser16necsys}, one can apply the rank-one update on the matrix of effective resistances. As a result, we can update the effective resistances of all links in order of $\mathcal O(n^2)$. More specifically, the matrix of effective resistances is given by 
\begin{equation}
R(L^m):= \mathbbm{1}_n \,  \diag \big( L^{\dag, m} \big) + \diag \big( L^{\dag, m} \big) \mathbbm{1}_n^{\text T} - 2 L^{\dag, m},
\label{R_matrix}
\end{equation}
for $m \in \{1,2,3\}$, where  $R(L^m)_{ij}=r_{\{i,j\}}(L^m)$. The update rule \eqref{R_matrix} can be obtained by substituting the rank-one update of $(L+L_e)^{\dag}$ from \eqref{eq:522} in \eqref{R_matrix} and
the $m$-th power of the rank-one update can be calculated in $\mathcal O(n^2)$ as it can be cast as only matrix-vector products.
Using these facts and the result of Theorem \ref{coro:1241}, the computational cost of \eqref{dif} for systemic performance measures $\zeta_1$, $\zeta_2$, and $\upsilon$ can be significantly reduced; more specifically,  the computational complexity of our algorithm reduces to
\[ \mathcal O \left (  \underbrace{n^3}_{\text{calculating}~ L^{\dag,m} {\text{'s at the beginning}} } +  \underbrace{ n^2}_{\text{rank-one update} } \times \underbrace{k}_{\text{number of steps}}\right).\]

\begin{figure}[t]
  \begin{center}  
    \includegraphics[trim = 60 30 30 30, clip,width=0.25\textwidth]{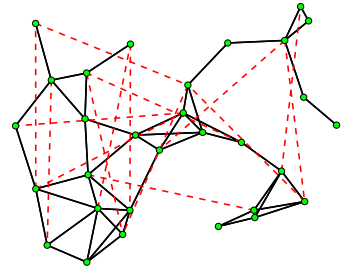}
   \end{center}
   \caption{\small{This is the coupling graph of the network in Example \ref{comparing} with $30$ nodes, where  the graph has $50$ original (black) links and the candidate set includes  all $15$ dashed red line segments.}}
   \label{Fig:candidate-coupling}
   \vspace{-0.55cm}
\end{figure}

For a generic systemic performance measure $\rho: \LL \rightarrow \R$, according to Theorem \ref{thm:schur-convex},  calculating its value requires knowledge of all Laplacian eigenvalues of the coupling graph. It is known that the  eigenvalue problem for symmetric matrices requires  $\mathcal O (n^{2.376} \log n)$ operations \cite{Yau93}. Suppose that calculating the value of spectral function $\Phi: \R^{n-1} \rightarrow \R$ in Theorem  \ref{thm:schur-convex} needs $\mathcal O \left(M_{\Phi}(n)\right)$ operations. Thus, the value of systemic performance measure $\rho(L)$ in equation \eqref{spectral-rho}, and similarly \eqref{dif}, can be calculated in $\mathcal O (n^{2.376} \log n + M_{\Phi}(n) )$. Based on this analysis, we conclude that the complexity of the greedy algorithm in Table \ref{greedy-table} is at most 
\[\mathcal O\left( \left (n^{2.376} \log n + M_{\Phi}(n) \right) \left (p -\frac{k-1}{2} \right )k\right).\]

\begin{figure*}
\centering
\subfloat[Spectral zeta function $\zeta_1$]{
	       \psfrag{x}[c][c]{\footnotesize $k$}        
       \psfrag{y}[c][c]{\footnotesize $\pi_k$}     
  \includegraphics[width=60mm]{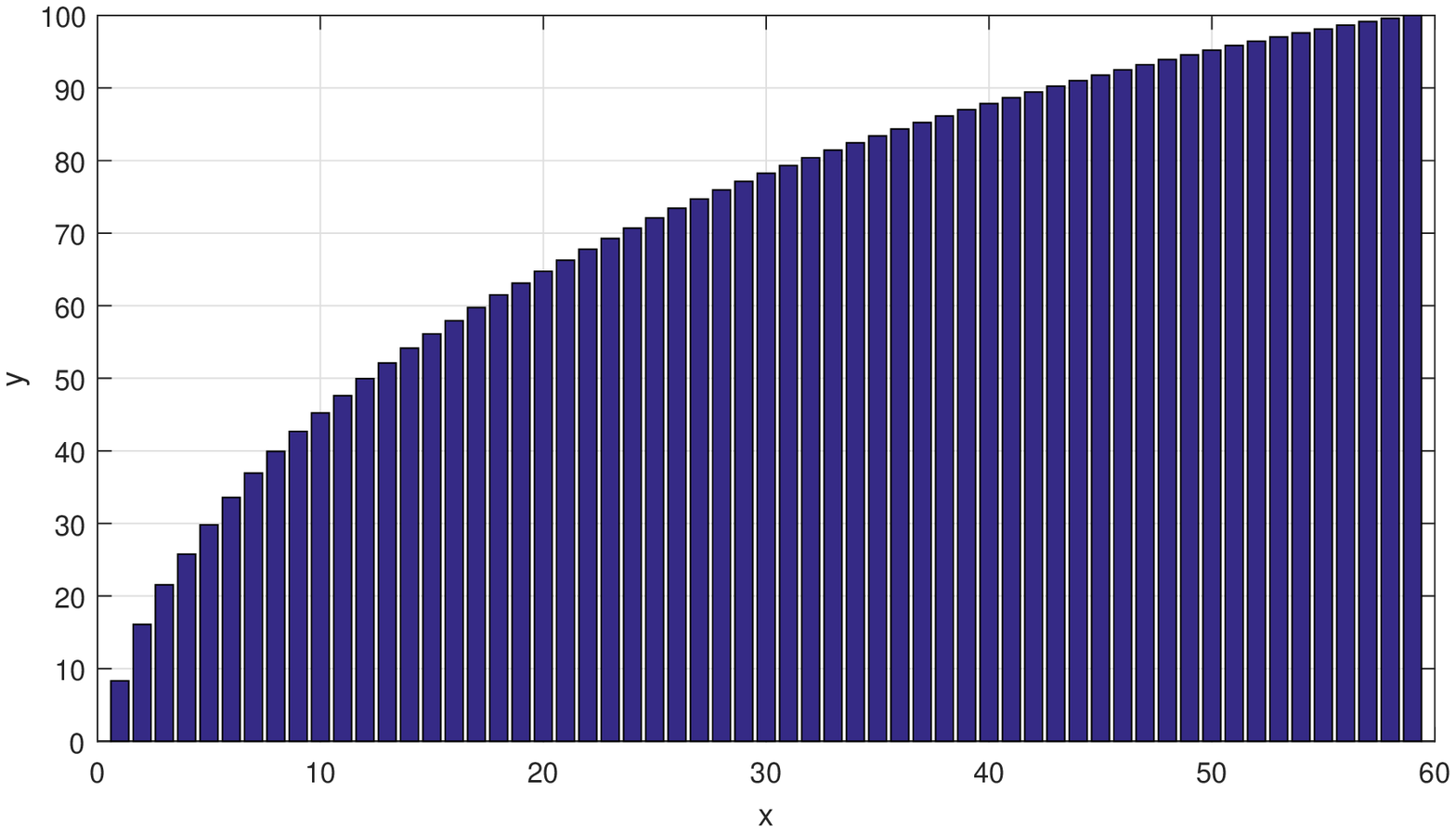}
}
\subfloat[Spectral zeta function  $\zeta_2$]{
	       \psfrag{x}[c][c]{\footnotesize $k$}        
       \psfrag{y}[c][c]{\footnotesize $\pi_k$}     
  \includegraphics[width=60mm]{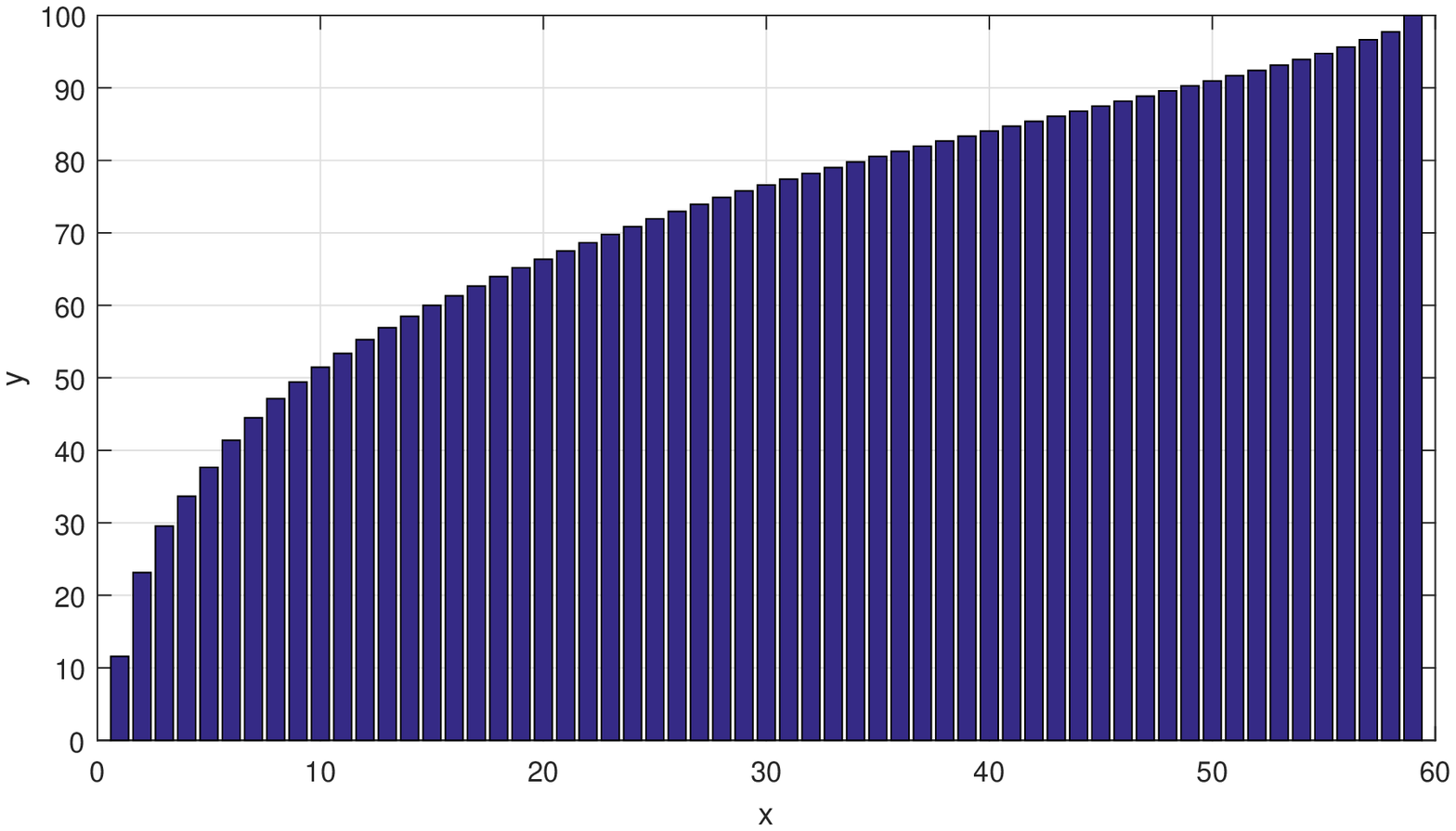}
}
\hspace{0mm}
\subfloat[Expected transient covariance $\tau_t$ where $t= 1$]{
	       \psfrag{x}[c][c]{\footnotesize $k$}        
       \psfrag{y}[c][c]{\footnotesize $\pi_k$}     
  \includegraphics[width=60mm]{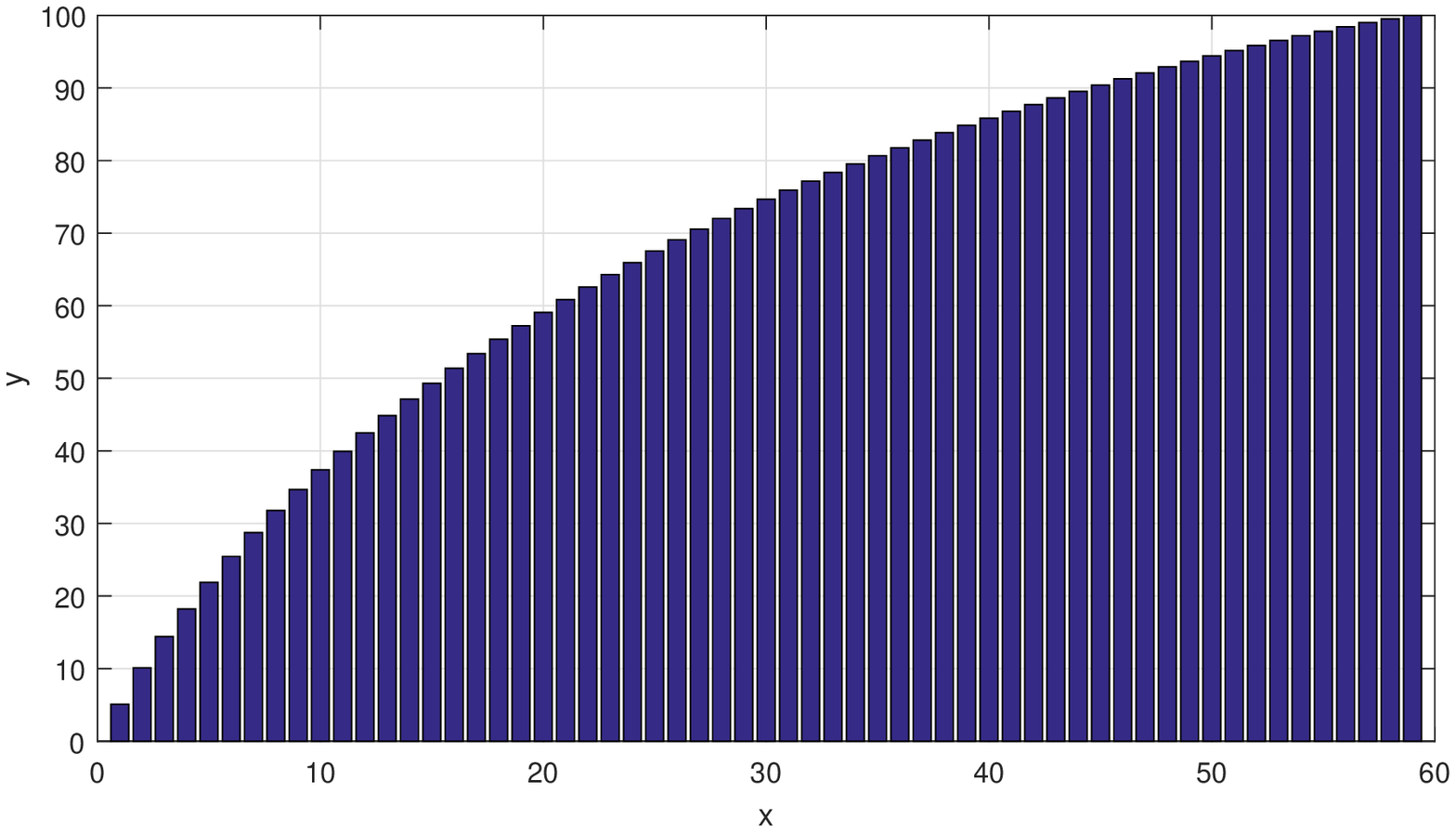}
}
\subfloat[$\gamma$-entropy $I_\gamma(.)$ where $\gamma = 2$]{
	       \psfrag{x}[c][c]{\footnotesize $k$}        
       \psfrag{y}[c][c]{\footnotesize $\pi_k$}     
  \includegraphics[width=60mm]{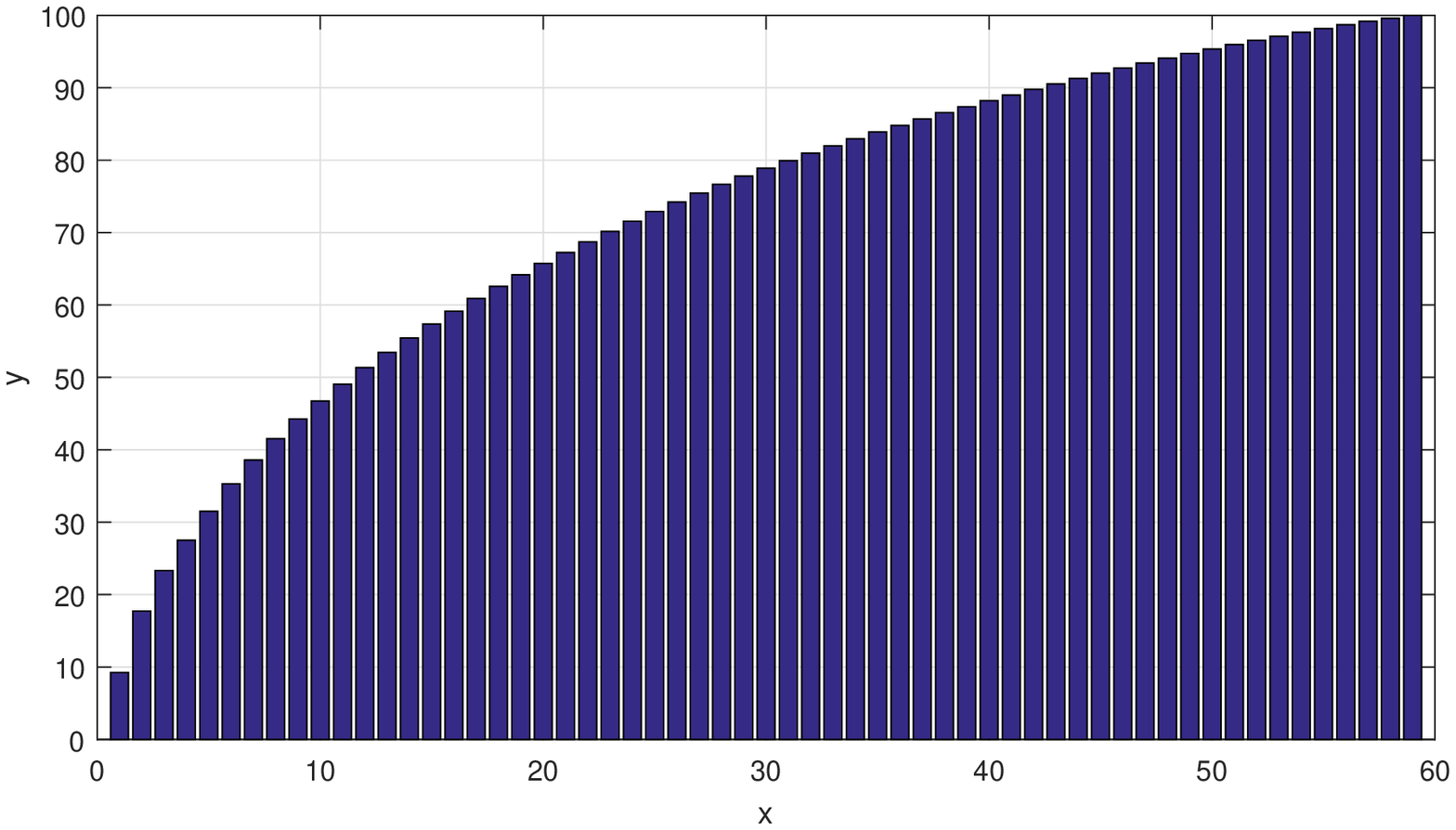}
}
\caption{\small These plots are discussed in Example \ref{ex:4}.  }
  	\label{fig:882}
	\vspace{-0.5cm}
\end{figure*}

\section{Numerical Simulations} \label{sec:simu}
In this section, we support our theoretical findings by means of some numerical examples. 

\begin{example}\label{ex:2}
This example investigates sensitivity of location of an optimal link as a function of its weight. Let us consider a linear consensus network  \eqref{first-order}-\eqref{first-order-G}, whose coupling graph is shown in Fig. \ref{fig:1618}, endowed by systemic performance measure \eqref{zeta-measure} with $q=1$. The graph shown in Fig. \ref{fig:1618} is a generic unweighted connected graph with $n=50$ nodes and $100$ links. We solve  the network synthesis problem \eqref{1-link} for  the candidate set with $|\mathcal E_c|=\frac{1}{2}n(n-1)$ that covers all possible locations in the graph. It is assumed that all candidate links have an identical weight $\varpi_0$. We use our rank-one update method in Theorem \ref{coro:1241} to study the effect of $\varpi_0$ on location of the optimal link. In Fig. \ref{fig:1618}(c), we observe that by increasing $\varpi_0$, the optimal location changes. When $\varpi_0=1$,  our calculations reveal that the optimal link in Fig. \ref{fig:1618}(a), shown by a blue line segment, maximizes $r_e(L^2)$ among all possible candidate links in set $\mathcal E_c$. By increasing the value of our design parameter  to $\varpi_0=1.2$ in Fig. \ref{fig:1618}(b), we observe that the location  of the optimal link moves. In our last scenario in Fig. \ref{fig:1618}(c), by setting $\varpi_0=1.6$, the optimal link moves to a new location that maximizes quantity  $r_e(L^2)/r_e(L)$ among all possible candidate links.	
\end{example}

\begin{example} \label{ex:1}
The usefulness of our theoretical fundamental hard limits in Theorem \ref{w-thm} in conjunction with our results in Theorem \ref{coro:1241} is illustrated in Fig. \ref{Fig:740}. Suppose that a  linear consensus network  \eqref{first-order}-\eqref{first-order-G} with a generic coupling graph with $n=60$, as shown in  Fig. \ref{Fig:751}(a), is given.  Let us consider the network design problem \eqref{1-link} with systemic performance measure \eqref{zeta-measure} for $q = 1$. The set of candidate links is the set of all possible links in the coupling graph, i.e.,  $|\mathcal E_c|=\frac{1}{2}n(n-1)$, where it is assumed that all candidate links have an identical weight $\varpi_0=20$.  Our goal is to compare optimality of our low-complexity update rule against brute-force search over all $|\mathcal E_c|=1770$ possible augmented graphs. The value of the systemic performance measure for each candidate graph is marked by blue star in Fig. \ref{Fig:740}. In this plot, the black circle highlights the value of performance measure for the network resulting from the rank-one search \eqref{1-link-B}. The red dashed line in Fig. \ref{Fig:740} shows the best achievable value for $\zeta_1$ according to Theorem \ref{w-thm}. The value of this hard limit can be  calculated merely using Laplacian eigenvalues of the original graph shown in Fig. \ref{Fig:751}(a). The location of the optimal link is shown in Fig. \ref{Fig:751}(b). One observes from Fig. \ref{Fig:740} that our theoretical fundamental limit justifies near-optimality of our rank-one update strategy \eqref{1-link-B} for networks with generic graph topologies.  
\end{example}

\begin{figure*}
\centering
  		\begin{tabular}{c c c}
      {\psfrag{x}[c][c]{\footnotesize $k$}
      \psfrag{y}[c][c]{\footnotesize $\zeta_1$}       
       \psfrag{1}[c][c]{\footnotesize}  
       \includegraphics[width=50mm]{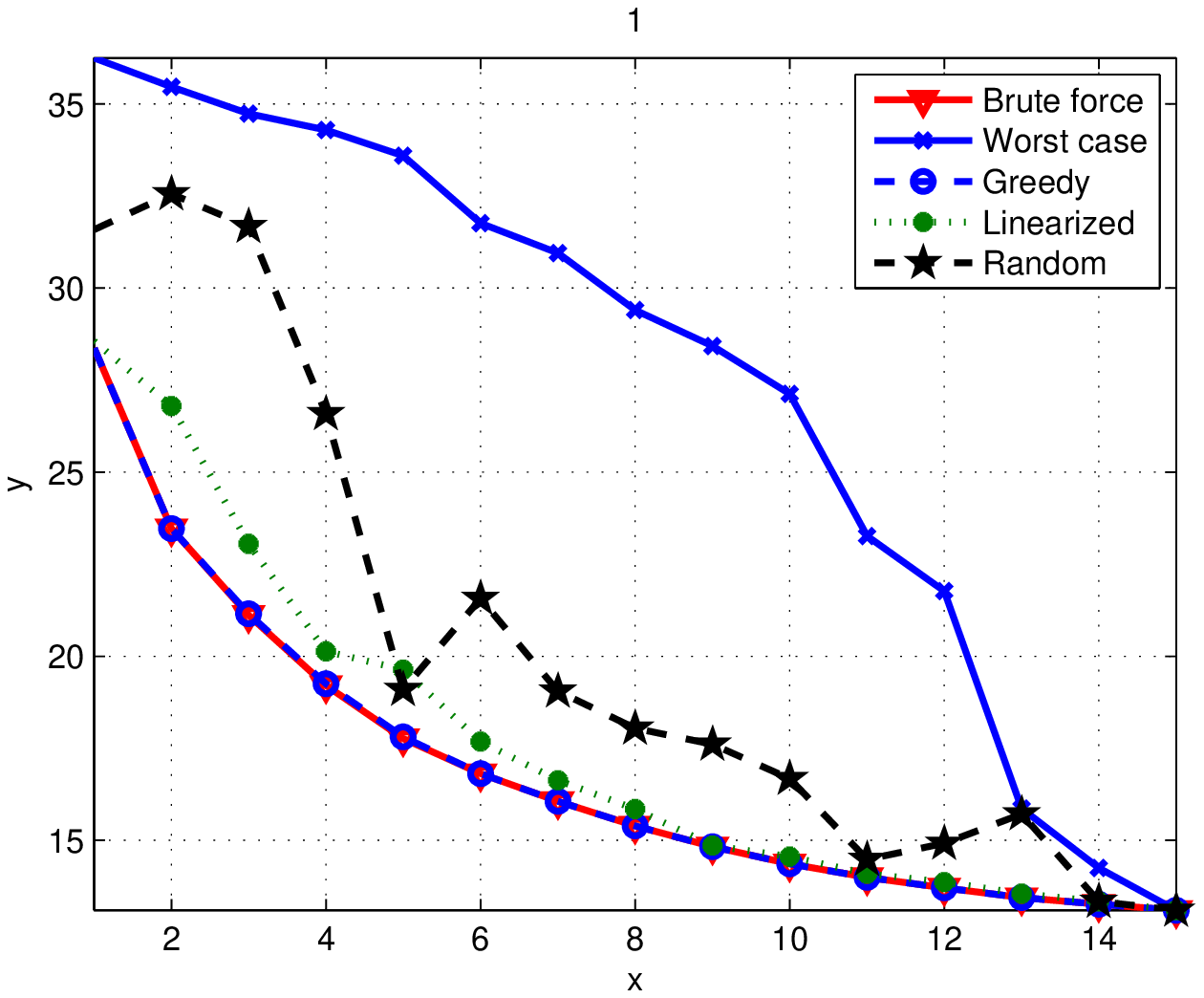}} &  {\psfrag{x}[c][c]{\footnotesize $k$}        
       \psfrag{y}[c][c]{\footnotesize $\zeta_2$}
        \psfrag{2}[c][c]{\footnotesize }    
       \includegraphics[width=50mm]{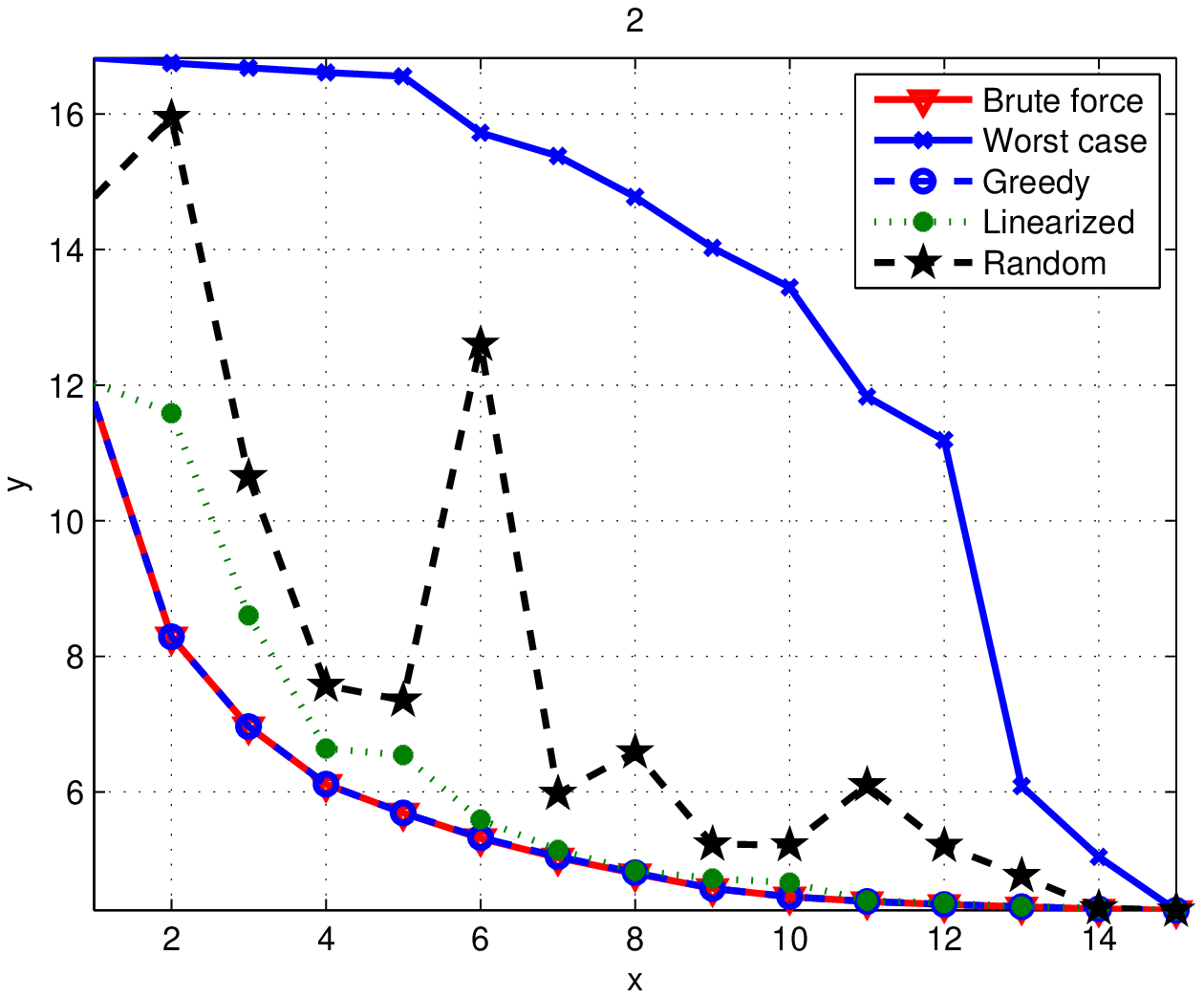}} &   {
       \psfrag{x}[c][c]{\footnotesize $k$}        
       \psfrag{y}[c][c]{\footnotesize $\eta$}  
        \psfrag{0}[c][c]{\footnotesize }  
  \includegraphics[width=50mm]{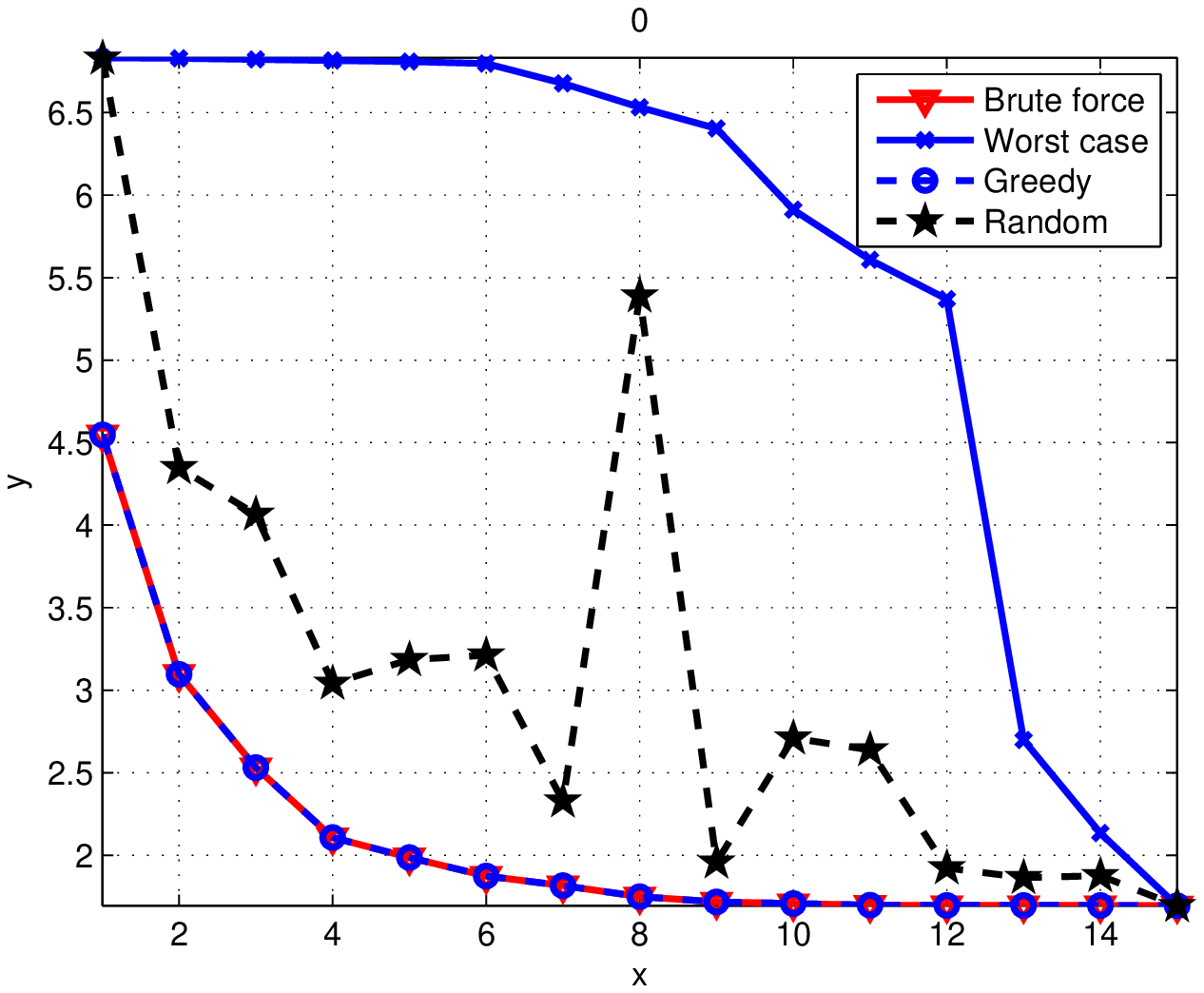}
} \\ (a)   \small Spectral zeta function $\zeta_1$& (b) \small Spectral zeta function $\zeta_2$& (c)  \small Hankel norm $\eta$\\
    {\psfrag{x}[c][c]{\footnotesize $k$}        
       \psfrag{y}[c][c]{\footnotesize $I_\gamma$}  
        \psfrag{5}[c][c]{\footnotesize }  
  \includegraphics[width=50mm]{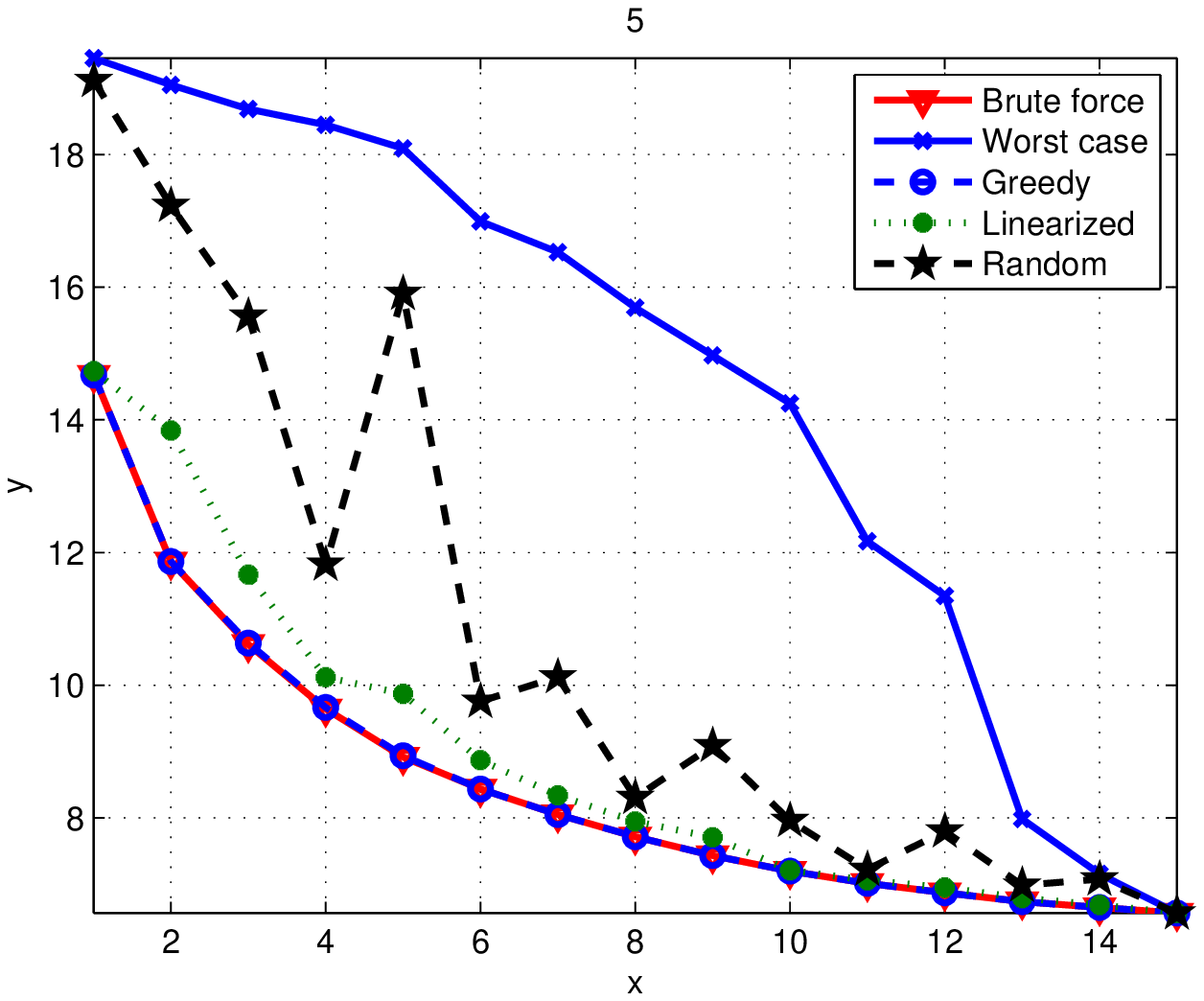}} &{
       \psfrag{x}[c][c]{\footnotesize $k$}        
       \psfrag{y}[c][c]{\footnotesize $\upsilon$}
        \psfrag{3}[c][c]{\footnotesize }    
  \includegraphics[width=50mm]{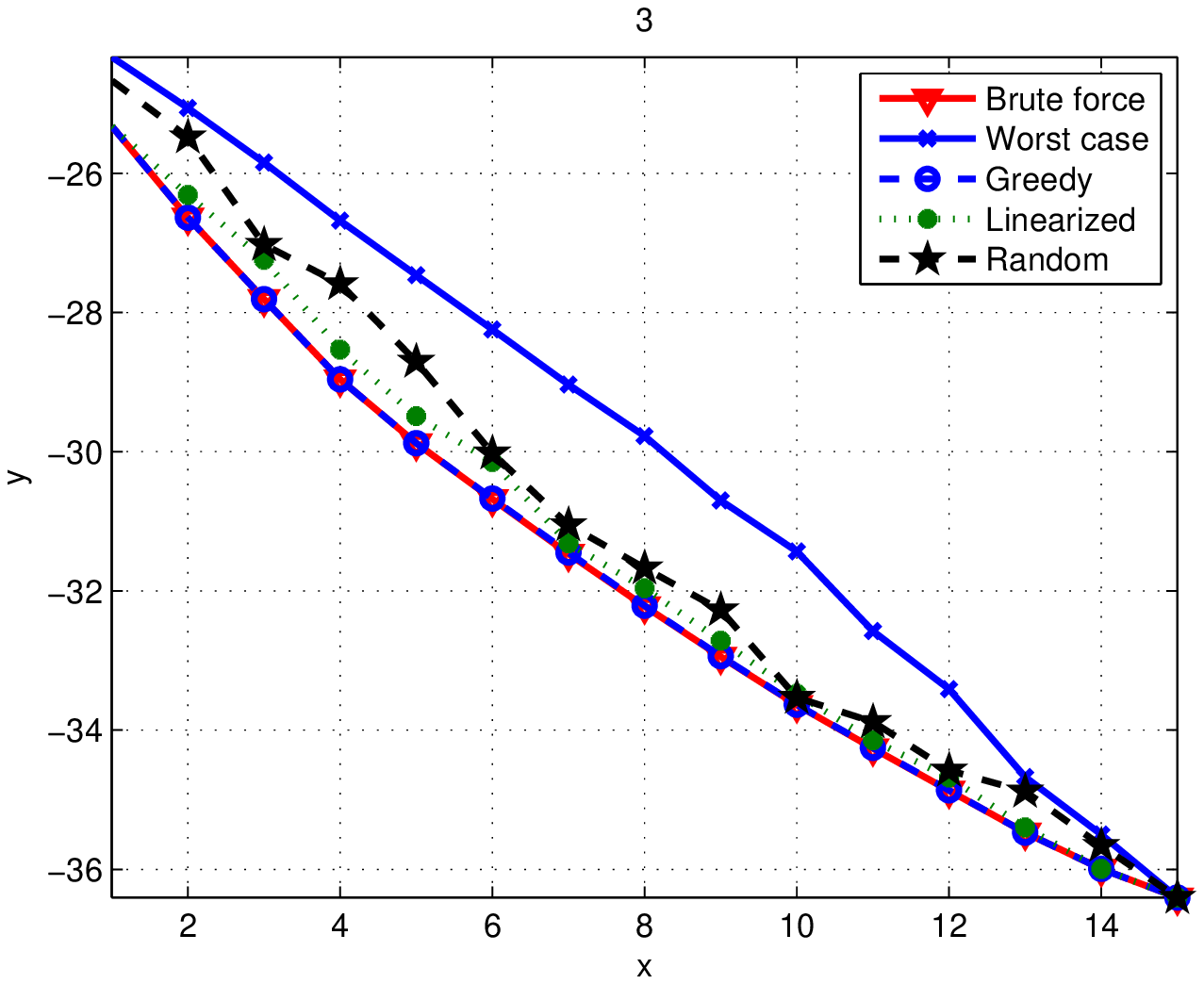}
}&{\psfrag{x}[c][c]{\footnotesize $k$}        
       \psfrag{y}[c][c]{\footnotesize $\tau_t$} 
        \psfrag{4}[c][c]{\footnotesize }   
  \includegraphics[width=50mm]{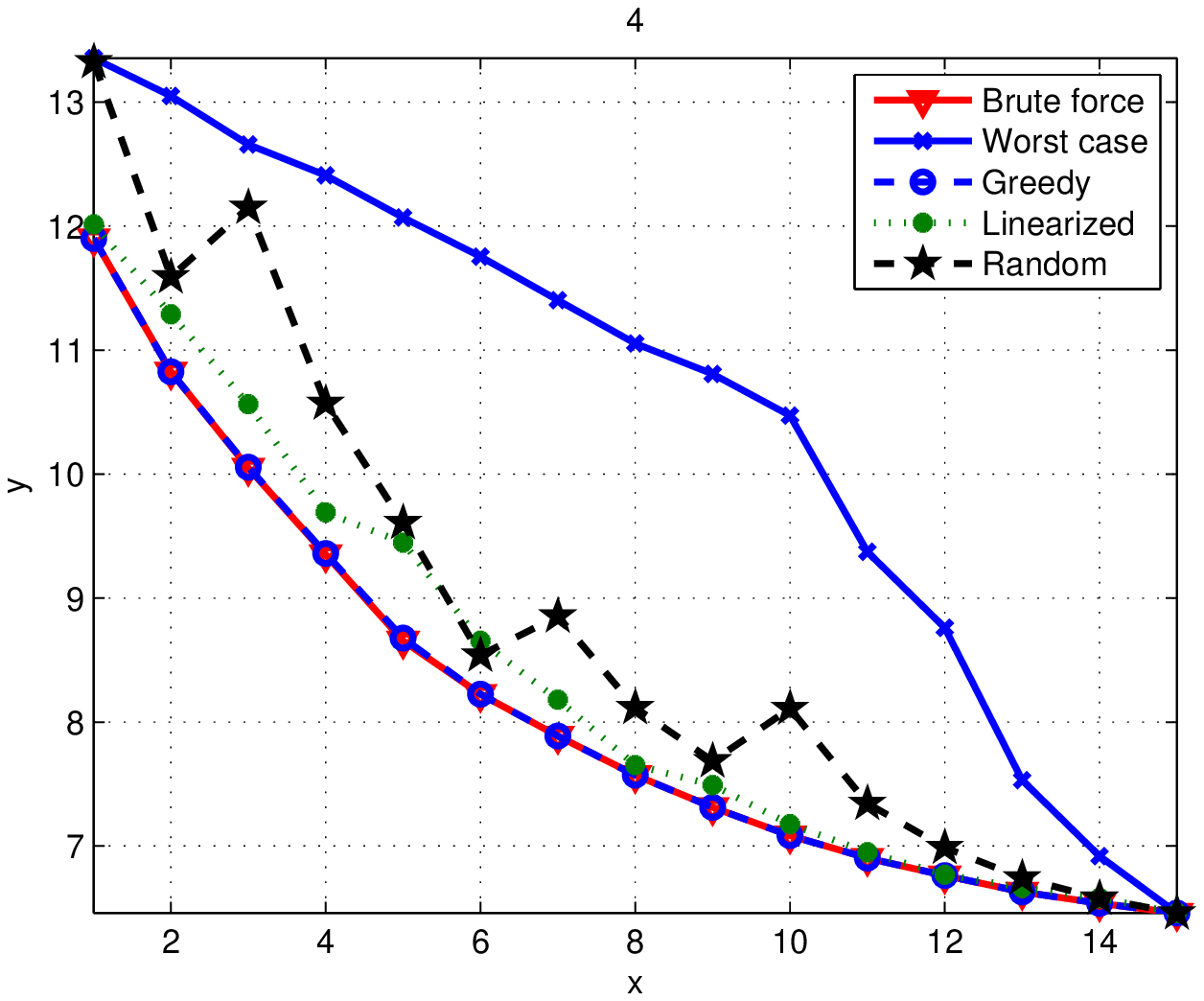}} \\
   (d) \small $\gamma$-entropy $I_\gamma$ where $\gamma = 20$ & (e) \small Uncertainty Volume $\upsilon$ & (f) \small Expected output covariance $\tau_t$ where $t=10$
    		\end{tabular}

\caption{\small These plots compare optimality gaps of five different methods for solving the network synthesis problem \eqref{k-link} in Example \ref{comparing}.}
 \label{Fig:2423}
 \vspace{-.9cm}
\end{figure*}

\begin{example}\label{ex:4}
This example follows up on our discussion at the end of Section \ref{sec:672}, where it is explained that the result of Theorem \ref{w-thm} can be utilized to choose reasonable values for design parameter $k$ in the network design problem \eqref{k-link}. We explain the procedure by considering a linear consensus network \eqref{first-order}-\eqref{first-order-G} with a given coupling graph by Fig. \ref{Fig:751}(a). The value of the lower bound (i.e., hard limit) in \eqref{fund-limit-1} is used to form the following quantity 
\[ \pi_k := \frac{\varrho_0 - \varrho_k}{\varrho_0} \times 100\]
that represents the percentage of performance enhancement for all values of parameter $1 \leq k \leq n-1$. Fig. \ref{fig:882} illustrates the value of $\pi_k$ with respect to four systemic performance measures: $\zeta_1$, $\zeta_2$, $\tau_t$ and $I_\gamma$. Depending on the desired level of performance, one can compute a sensible value for design parameter $k$ merely by looking up at the corresponding plots. For instance, in order to achieve $50 \%$ performance improvement, one should add at least $13$, $10$, $16$, and $12$ weighted links with respect to $\zeta_1$, $\zeta_2$, $\tau_t$ and $I_\gamma$, respectively. We verified tightness of this estimate by running our greedy algorithm in Table \ref{greedy-table}, where the candidate set is equal to the set of all possible links with identical weight $10$. Our simulation results reveal that by adding $13$, $10$, $16$, and $12$ links from the candidate set, the network performance improves by $40.60\%$, $45.10\%$,  $37.76\%$, and $40.61\%$ with respect to $\zeta_1$, $\zeta_2$, $\tau_t$, and $I_\gamma$, respectively. Our theoretical bounds predict that network performance can be further improved  by increasing weights of the candidate links. In our example, if we increase the weight from $10$ to $500$, the network performance boosts by more than $46\%$ for all mentioned systemic performance measures. 
\end{example}

\begin{example} \label{comparing}
We compare optimality gaps of our proposed greedy (see Table \ref{greedy-table}) and linearization-based (see Table \ref{table-linear}) methods versus brute-force and simple-random-sampling methods. The brute-force method runs an exhaustive search to find the global optimal solution of problem \eqref{k-link}; however, it cannot be used for medium to large size networks. In order to make our comparison possible, we consider a linear consensus network \eqref{first-order}-\eqref{first-order-G} with $n=30$ nodes over the graph shown in Fig. \ref{Fig:candidate-coupling}. Weights of all links, both in the coupling graph and the candidate set, are equal to $1$. Our control objective is to solve the network synthesis problem \eqref{k-link}, where the candidate set consists of $15$ links that are shown by red-dashed lines in Fig. \ref{Fig:candidate-coupling}. The outcome of our simulation results are explicated  in Fig. \ref{Fig:2423}, where we run our algorithms and compute the corresponding values for systemic performance measures for all $k=1,\ldots, 15$. One observes that our greedy algorithm performs nearly as optimal as the brute-force method. This is mainly due to convexity and monotonicity properties of the class of systemic performance measures that enable the greedy algorithm to produce near-optimal solutions with respect to this class of measures. As one expects, our greedy algorithm outperforms our linearization-based  method. It is noteworthy that  the time complexity of the linearization method is comparably less than the greedy algorithm. The usefulness of the linearization-based method  accentuates itself when weight of candidate links are small and/or $k$ is large. 

\end{example}

\section{Discussion and Conclusion}
\label{sec:913}
In the following, we provide explanations for some of the outstanding and remaining problems related to this paper. 

\vspace{0.1cm}
\noindent {\it Convex Relaxation:}
The constraints of the combinatorial problem \eqref{k-link} can be relaxed by allowing the link weights to vary continuously. The relaxed problem will be a spectral convex optimization problem \cite{lewis}. In some special cases, such as when the cost function is $\zeta_1$ or $\zeta_2$, the relaxed problem can be equivalently cast  as a semidefinite programming problem \cite{Siami14acc, Siami14cdc-2}. However, for a generic systemic performance measure, we need to develop some low-complexity specialized optimization techniques to solve the corresponding spectral optimization problem, which is beyond the scope of this paper. 

\vspace{0.1cm}
\noindent {\it Higher-Order Approximations:}
In Subsection \ref{subsec:B}, we employed the first-order approximation of a systemic performance measure. One can easily extend our algorithm by considering second-order approximations of a systemic performance measure in order to gain better optimality gaps.

\vspace{0.1cm}
\noindent {\it Non-spectral Systemic Performance Measures:} The class of spectral systemic performance measures can be extended to include non-spectral measures as well. This can done by relaxing and replacing the orthogonal invariance property by permutation invariance property. The local deviation error is an example of a non-spectral systemic performance measure \cite{Siami14cdc-2, Siami15necsys}. Our ongoing research involves a comprehensive treatment of this class of measures.  

\appendix[Notable Classes of Systemic Performance Measures]
\label{ssec:example}

	In the following, we will revisit several existing and widely-used examples of performance measures in linear consensus networks and prove that they are indeed systemic performance measures according to the definition. 

\subsection{Sum of Convex Spectral Functions}

	This class of performance measures is generated  by forming summation of a given function of non-zero Laplacian eigenvalues. 

\begin{theorem}\label{f-sum}
For a given matrix $L \in \mathfrak L_n$, suppose that  {$\varphi: \R_{+} \rightarrow \R$} is a decreasing convex function. Then, the following spectral function 
\begin{equation}
	\rhoo (L) ~=~\sum_{i=2}^n \varphi(\lambda_i)
	\label{eq:512}
\end{equation}
is a systemic performance measure. Moreover, if {$\varphi$} is also a homogeneous function of order $-\kappa$ with $\kappa>1$,  then the following spectral function 
\begin{equation}
	\rhoo(L) ~=~ \left ( \sum_{i=2}^n \varphi(\lambda_i) \right)^{\frac{1}{\kappa}} 
	\label{eq:measure-normalized}
\end{equation}
is also a systemic performance measure.
\end{theorem}
\begin{proof}
{First we show that measure \eqref{eq:512} is monotone with respect to the positive definite cone. If we assume that $L_{2} \preceq L_{1}$, then based on Theorem A.1 in \cite[Sec. 20]{marshall11}, it follows that 
\begin{equation}
	\lambda_i({L_2}) ~\leq~ \lambda_i({L_1}),~~\text{for}~~i=1,2,\cdots,n.\label{eq:527}
\end{equation}
Thus, using \eqref{eq:527} and the fact that $\varphi(.)$ is decreasing, we get the monotonicity property of measure \eqref{eq:512}. Also, it is not difficult to show that measure \eqref{eq:512} satisfies Property 2. To do so, let $L_1$ and $L_2$ be two Laplacian matrices in $\mathfrak L_n$. Recall that $\lambdaa(L_i)$, $i=1,2$ is the vector of eigenvalues of $L_i$ in ascending order. According to Theorem G.1 in \cite[Sec. 9]{marshall11}, we know that
\begin{equation}
	\lambdaa(\alpha L_1 + (1-\alpha) L_2) ~\unlhd~ \alpha \lambdaa(L_1) + (1-\alpha)\lambdaa(L_2),
	\label{con-eig}
\end{equation}
for every $0 \leq \alpha \leq 1$, and $\unlhd$ denotes the majorization preorder \cite{marshall11}. Besides, we note that based on Proposition. C.1 in \cite[Sec.3]{marshall11}, measure \eqref{eq:512} is a Schur-convex function. Consequently, using this property and \eqref{con-eig}, we have
\begin{eqnarray}
	&&\hspace{-.7cm}\rhoo( \alpha L_1 + (1-\alpha) L_2) ~=~ \sum_{i=2}^n\varphi \left (\lambda_i(\alpha L_1 + (1-\alpha) L_2)\right) \nonumber \\
	&&~~~~~~\leq~ \sum_{i=2}^n \varphi \big (\alpha \lambda_i( L_1) + (1-\alpha) \lambda_i(L_2)\big). 
	\label{eq:414}
\end{eqnarray}
From \eqref{eq:414} and the desired convexity property of $\varphi(.)$, we get the convexity property as follows
\begin{eqnarray*}
	&&\hspace{-.5cm} \rhoo( \alpha L_1 + (1-\alpha) L_2) ~\leq~ \sum_{i=2}^n \varphi \left (\alpha \lambda_i( L_1) + (1-\alpha) \lambda_i(L_2)\right) \\	
	&&~~~~~~~~~~~\leq~   \alpha  \sum_{i=2}^n \varphi \left (\lambda_i( L_1)\right) + (1-\alpha)  \sum_{i=2}^n \varphi \left (\lambda_i(L_2)\right) \\
	&&~~~~~~~~~~~=~ \alpha \rhoo (L_1)+(1-\alpha) \rhoo (L_2),
 \end{eqnarray*}
for every $0 \leq \alpha \leq 1$. Finally, systemic measure \eqref{eq:512} is orthogonal invariant because it is a spectral function. Hence, measure \eqref{eq:512} satisfies all properties of Definition \ref{def-schur-systemic}. This completes the proof of first part.
\\
Next, we show that measure \eqref{eq:measure-normalized}  satisfies Properties 1, 2, and 3 given by Definition \ref{def-schur-systemic}. Similar to the previous case, it is straightforward to verify that measure \eqref{eq:measure-normalized} has Property 1. Now we show that measure \eqref{eq:measure-normalized} has Property 2, i.e., it is a convex function over the set of Laplacian matrices. 
By hypothesis, $\varphi(.)$ is a homogeneous function of order $-\kappa$, therefore, we have 
\begin{equation}
	\varphi(\lambda_i)=\lambda_i^{-\kappa}\varphi(1).
	\label{eq:438}
\end{equation}
Using \eqref{eq:438} and \eqref{eq:measure-normalized}, we get
\begin{equation}
	\rhoo(L)=K\left(\sum_{i=2}^n \lambda_i^{-\kappa}\right)^{\frac{1}{\kappa}},
	\label{eq:445}
\end{equation}
where $K=\sqrt[\kappa]{\varphi(1)}$. It is well-known function \eqref{eq:445} is convex for $\lambda_i > 0$ where $i=2, \ldots, n$ and $\kappa >1$. Based on the proof of Part (i), measure $\rhoo^{\kappa}(.)$ is a Schur-convex function. Consequently, we get 
\begin{eqnarray}
	&&\hspace{-.5cm} \rhoo( \alpha L_1 + (1-\alpha) L_2)\,\leq\, \nonumber \\
	&&~~~ K\left (  \sum_{i=2}^{n} \big (\alpha \lambda_i( L_1) + (1-\alpha) \lambda_i(L_2)\big)^{-\kappa} \right)^{\frac{1}{\kappa}}.
	\label{eq:451}
\end{eqnarray}
Now using \eqref{eq:451} and the convexity of \eqref{eq:445} with respect to $\lambda_i$'s, we have
 \begin{eqnarray}
	&&\hspace{-.5cm} \rhoo( \alpha L_1 + (1-\alpha) L_2) \nonumber \\
	&&~~~~\leq ~ K\left (  \sum_{i=2}^{n} \big (\alpha \lambda_i( L_1) + (1-\alpha) \lambda_i(L_2)\big)^{-\kappa} \right)^{\frac{1}{\kappa}} \nonumber \\
	&&~~~~\leq ~ \alpha \rhoo(L_1)+(1-\alpha) \rhoo(L_2). \nonumber 
	\label{eq:451}
\end{eqnarray}
This completes the proof.}
\end{proof}

There are several important examples of performance measures that belong to this class. 

\hspace{.0cm}
\subsubsection{Spectral Zeta Functions}\label{zeta}
For a given network  \eqref{first-order}-\eqref{first-order-G}, its corresponding spectral zeta function of order $q \geq 1$ is defined by
\begin{equation}
	{\zeta}_{q}(L)~:=~\bigg( \sum_{i=2}^n \lambda_i^{-q} \bigg)^{1/q},
	\label{zeta-measure}
\end{equation}
where $\lambda_2, \ldots,\lambda_n$ are eigenvalues of $L$ \cite{Hawking}. According to Assumption \ref{assum-coupling-graph}, all the Laplacian eigenvalues $\lambda_2, \ldots,\lambda_n$ are strictly positive and, as a result, function \eqref{zeta-measure} is well-defined. The spectral zeta function of a graph captures all its spectral features. In fact, it is straightforward to show that  every two graphs in $\LL$ with identical zeta functions for all parameters $q \geq 1$ are isospectral \footnote{This is because for a given graph with $n$ nodes, Laplacian eigenvalues $\lambda_2, \ldots, \lambda_n$ can be uniquely determined by using equation \eqref{zeta-measure} and having the value of ${\zeta}_{q}(L)$ for $n-1$ distinct values of $q$. We refer to algebraic geometric tools for existing algorithms to solve this problem \cite{Sturmfels02}.}. 

Since $\varphi(\lambda)=\lambda^{-q}$ for $q \geq 1$ is a decreasing convex function, the spectral function \eqref{zeta-measure} is a systemic performance measure according to Theorem \ref{f-sum}. The systemic performance measure $\frac{1}{2}{\zeta}_{1}(L)$ is equal to the $\HH_2$-norm squared of a first-order consensus network \eqref{first-order}-\eqref{first-order-G} and ${\frac{1}{\sqrt{2}}\zeta}_{2}(L)$ equal to the $\HH_2$-norm of a second-order consensus model of a network of multiple agents (c.f. \cite{Siami13cdc}).
\subsubsection{Gamma Entropy}

The notion of gamma entropy arises in various applications such as {the design of minimum entropy controllers and interior point polynomial-time methods
in convex programming} with matrix norm constraints \cite{Blondel1999}. As it is shown in \cite{boyd97}, the notion of gamma entropy can be interpreted as a performance measure for linear time-invariant systems with random feedback controllers  by relating the gamma entropy to the mean-square value of the closed-loop gain of the system.

%
\begin{definition}
The $\gamma$-entropy of network \eqref{first-order}-\eqref{first-order-G} is defined as
\begin{eqnarray*}
	\Scale[.85]{ I_{\gamma}(L) := \begin{cases}
	\frac{-\gamma^2}{2\pi}\int_{-\infty}^{\infty} \log \det \big(I- \gamma^{-2} G(j\omega)G^*(j\omega) \big)d\omega~~\text{for}~ \gamma \geq \|G\|_{\mathcal H_\infty}  \\
	\\
	\infty~~~~~~~~~~~~~~~~~~~~~~~~~~~~~~~~~~~~~~~~~~~~~~~~~~~\text{otherwise}
	\end{cases}}
	\label{gamma-formula}
\end{eqnarray*}
where $G(j\omega)$ is the transfer function of network \eqref{first-order}-\eqref{first-order-G} from $\xi$ to $y$.
\end{definition}
\hspace{.2cm}
\begin{theorem}
For a given linear consensus network \eqref{first-order}-\eqref{first-order-G}, the value of the $\gamma$-entropy can be explicitly computed in terms of network's Laplacian eigenvalues as follows
\begin{eqnarray}
	\Scale[1]{
	I_{\gamma}(L)=\begin{cases}
	\displaystyle \sum_{i=2}^n f_{\gamma}(\lambda_i)~~~~~\gamma \geq \lambda_2^{-1}  \\
	\\
	\infty~~~~~~~~~~~~~~~\text{otherwise}
	\end{cases}}
	\label{gamma-formula}
\end{eqnarray}
where $f_{\gamma}(\lambda_i)=\gamma^2 \left( \lambda_i- \left(\lambda_i^2-\gamma^{-2}\right)^{\frac{1}{2}} \right)$.
Moreover, the $\gamma$-entropy $I_{\gamma}(L)$ is a systemic performance measure.
\end{theorem}

\begin{proof}
{ First we obtain the transfer function of network \eqref{first-order}-\eqref{first-order-G} from $\xi$ to $y$. In order to do that, let us rewrite the network in its disagreement form  \eqref{first-order-d}-\eqref{first-order-G-d} (see \cite{Siami14cdc-1} for more details).
Then, it follows that
\begin{eqnarray}
	&&\hspace{-.3cm}G(s) \,=\,  M_n\left (sI_n+L+\frac{1}{n}J_n\right )^{-1}M_n\nonumber \\
	&&~~~\,=\, M_nU\diag \left [\frac{1}{s+1},\frac{1}{s+\lambda_2},\cdots, \frac{1}{s+\lambda_n}\right ]U^{\text T}M_n\nonumber \\
	&&~~~\,=\, U\diag\left [0,\frac{1}{s+\lambda_2},\cdots, \frac{1}{s+\lambda_n}\right ]U^{\text T}, \label{G-form1}
\end{eqnarray}
where $U$ is the corresponding orthonormal matrix of eigenvectors of $L$.
Now, we calculate the $\gamma$-entropy by substituting the transfer function \eqref{G-form1} in \eqref{gamma-formula} as follows
\begin{eqnarray*}
	I_{\gamma}(G)&=&\frac{-\gamma^2}{2\pi}\int_{-\infty}^{\infty} \log \det \left(I_n- \gamma^{-2} G(j\omega)G^*(j\omega) \right)d\omega\\
	&=&\frac{-\gamma^2}{2\pi}\int_{-\infty}^{\infty} \log \det \left (I_n- \gamma^{-2} G(j\omega)G^*(j\omega) \right )d\omega.
\end{eqnarray*}
%
Then, using the fact that $UU^{\rm T}=I_n$ and \eqref{G-form1}, one can write:
\begin{equation}
	\log \det \left (I_n- \gamma^2 G(j\omega)G^*(j\omega)\right ) \,=\, \log \prod_{i=2}^n \left (1 - \frac{\gamma^{-2}}{\lambda_i^2+\omega^2} \right ).
	\label{eq:754}
\end{equation}
Moreover, we know that
\begin{equation}
	\int_{-\infty}^{\infty} \log \left (1 - \frac{\gamma^{-2}}{\lambda_i^2+\omega^2} \right)d\omega=-\gamma^2 \left( \lambda_i- \left(\lambda_i^2-\gamma^{-2}\right)^{\frac{1}{2}} \right),
	\label{eq:integ-ent}
\end{equation}
for $\gamma \geq \lambda_i^{-1}$.
Therefore, based on this result and \eqref{eq:754}, we get the desired result
\begin{eqnarray}
&& \hspace{-.4cm}-\sum_{i=2}^n\int_{-\infty}^{\infty} \log \left (1 - \frac{\gamma^{-2}}{\lambda_i^2+\omega^2}\right)d\omega \, = \, \nonumber \\
&&~~~~~~~~~~~~~~~~~~~~~~~~~~~~~\sum_{i=2}^n \gamma^2 \left( \lambda_i- \sqrt{\lambda_i^2-\gamma^{-2}} \right),
	\label{eq:entrop-form}
\end{eqnarray}
for $\gamma \geq \lambda_2^{-1}$.
Note that $f_\gamma(.)$ is a convex decreasing function in $[\gamma^{-1}, \infty)$, therefore, according to Theorem \ref{f-sum} and \eqref{eq:entrop-form}, the $\gamma$-entropy $I_{\gamma}(L)$ is a systemic performance measure.}
\end{proof}

The following result presents the connection between the $\gamma$-entropy measure and the $\mathcal H_2$-norm of the network.
\begin{theorem}
The following equality holds for the $\gamma$-entropy measure of network \eqref{first-order}-\eqref{first-order-G}
\[ \lim _{\gamma \rightarrow \infty}I_{\gamma}(L) ~=~ \frac{1}{2}\sum_{i=2}^n \lambda_i^{-1} ~=~ \| G \|_{\mathcal H_2}^2 ~=~ \lim_{t \rightarrow \infty} \E\big\{ y^{\text T}(t)y(t) \big\}, \]
where $G(.)$ is the transfer function of network \eqref{first-order}-\eqref{first-order-G}.
\end{theorem}
\begin{proof}
%
 Let us define $a= \gamma^{-1}$, then we have 
 \begin{equation*}
	\lim_{\gamma \rightarrow \infty}f_\gamma(x)=\lim_{\gamma \rightarrow \infty} \gamma^2 \left( x-\sqrt{x^2-\gamma^{-2}} \right)= \lim_{a \rightarrow 0}  \frac{x- \sqrt{x^2-a^{2}}}{a^2}.
\end{equation*}
Using L'Hopital rule, we get
 \begin{equation*}
	\lim_{a \rightarrow 0}  \frac{a\left(x^2-a^{2} \right)^{-\frac{1}{2}}}{2a}=\lim_{a \rightarrow 0}  \frac{\left(x^2-a^{2} \right)^{-\frac{1}{2}}}{2} = \frac{1}{2} x^{-1},
\end{equation*}
for all $x>0$ to prove that $\lim _{\gamma \rightarrow \infty}I_{\gamma}(L) = \frac{1}{2}\sum_{i=2}^n \lambda_i^{-1}$. Finally, we use \cite[Th. 1]{Siami14arxiv} to show that $\frac{1}{2}\sum_{i=2}^n \lambda_i^{-1}=\|G\|_{\mathcal{H}_2}^2=\lim_{t \rightarrow \infty} \E\big\{ y^{\text T}(t)y(t) \big\}$.
\end{proof}

\hspace{.1cm}
\subsubsection{Expected Transient Output Covariance}

	We consider a transient performance measure at time instant $t > 0$  that is defined by
\begin{equation}
	\tau_t(L) := \E \big\{ y^{\text T}(t)y(t) \big\},
 	\label{eq:transient}
\end{equation}
{where it is assumed that each $\xi_i(t)$ for all $t \geq 0$ is a  white Gaussian noise with zero mean and unit variance and all $\xi_i$'s are independent of each other}. 

In the following, we show that this performance measure is a spectral function of Laplacian eigenvalues. 
\begin{theorem}
\label{th:hankel}
	For a given linear consensus network \eqref{first-order}-\eqref{first-order-G}, the transient measure can be expressed as
\begin{equation}
	\tau_t(L)~=~ \sum_{i=2}^n \frac{1- e^{-\lambda_i t}}{2\lambda_i}.
	\label{transient-measure}
\end{equation}
Moreover, $\tau_t(L)$ is a systemic performance measure for all $t > 0$. 
\end{theorem}

\begin{proof}
The covariance matrix of the output vector is governed by the following matrix  differential equation 
\begin{equation}
	\dot Y(t) \,=\,-L \, Y(t)-Y(t) \, L + M_n,
	\label{eq:804}
\end{equation}
where $Y(t)=\cov(y(t),y(t))$. Using the closed-form  solution of \eqref{eq:804}, which is given by 
\begin{equation} 
	Y(t) \,=\, \int_{0}^{t} e^{-L\tau }M_n e^{-L\tau} d \tau, 
	\label{eq:809}
\end{equation}
we get 
\begin{eqnarray}
	\E \big\{ y^{\text T}(t)y(t)\big\}&=&\tr(Y(t)) = \tr \left(\int_{0}^{t} e^{-L \tau}M_n e^{-L \tau} d \tau \right)  \nonumber  \\
	&=&\sum_{i=2}^n \int_{0}^{t} e^{-2\lambda_i \tau} d \tau = \sum_{i=2}^n \frac{1- e^{-\lambda_i t}}{2\lambda_i}.
\end{eqnarray}
Since $f(x)=\frac{1- e^{-x t}}{2x}$ is convex and decreasing with respect to $x$ on $\R_+$, we can conclude that $\tau_t(L)$ is a systemic performance measure according to Theorem \ref{f-sum}.
\end{proof}

We note that when $t$ tends to infinity, the value of the transient performance measure becomes equal to the $\mathcal{H}_2$-norm squared of the network, i.e., $\tau_{\infty}(L) = \| G\|^2_{\mathcal{H}_2}$. 

\hspace{.0cm}
\subsubsection{Hankel Norm}

	The Hankel norm of a network with \eqref{first-order}-\eqref{first-order-G}  and transfer function $G(j \omega)$ from $\xi$ to $y$ is defined as the $\mathcal L_2$-gain from past inputs to the future outputs, \ie
\[\|G\|_{H}^2~:=~\sup_{\xi \in L_2 (-\infty, 0] } \frac{\int _0^{\infty} y^{\rm T}(t) y(t) dt}{\int_{-\infty}^0 \xi^{\rm T}(t) \xi(t) dt}. \]
The value of the Hankel norm of network \eqref{first-order}-\eqref{first-order-G} can be equivalently computed using the Hankel norm of its disagreement form \cite{olfati} that is given by 
\begin{eqnarray}
	\dot x_d(t)&=&-L_{d}  \, x_d(t) + M_n \, \xi(t), \label{first-order-d}\\
     	y(t)&=&M_n \hspace{0.05cm}x_d(t), \label{first-order-G-d}
\end{eqnarray}
where the disagreement vector is defined by 
\begin{equation}
	x_d(t) ~:=~ M_n \, x(t)~=~x(t) - \frac{1}{n}J_n \, x(t).
	\label{dis-vector-1}
\end{equation}
The disagreement network \eqref{first-order-d}-\eqref{first-order-G-d} is stable as every eigenvalue of the state matrix $-L_{d}=-(L+\frac{1}{n}J_n)$ has a strictly negative real part. One can  verify that the transfer functions from $\xi(t)$ to $y(t)$ in both realizations are identical. Therefore, the Hankel norm of the system from $\xi(t)$ to $y(t)$ in both representations are well-defined and equal, and is given by  \cite{hankel}
\begin{equation}
	\eta(L):=\|G\|_H= \sqrt{\lambda_{\max}(PQ)},
	\label{hankel}
\end{equation}
where the controllability Gramian $P$ is  the unique solution of
\begin{equation*}
	\Big(L + \frac{1}{n}J_n\Big)P+P\Big(L+\frac{1}{n}J_n \Big)- M_n =  0 
\end{equation*}
and the observability Gramian $Q$ is   the unique solution of
\begin{equation*} 
	Q\Big(L + \frac{1}{n}J_n\Big)+\Big(L+\frac{1}{n}J_n \Big)Q - M_n =  0. 
\end{equation*}
%

\begin{theorem}
\label{th:hankel}

The value of the Hankel norm of consensus network \eqref{first-order}-\eqref{first-order-G} is equal to 
\[  \eta(L) ~=~ \frac{1}{2}\lambda_2^{-1}\] 
and it is a systemic performance measure.
\end{theorem}

\begin{proof}
According to the definition \eqref{hankel}, we get
\begin{equation*}
 	\eta(L) ~=~\sqrt{\lambda_{\max}(PQ)}~=~ \sqrt{\lambda_{\max}\left((L{^\dag})^2\right)}~=~{\lambda_2^{-1}}.
\end{equation*}
Moreover, based on Theorem \ref{f-sum}, we know that the spectral zeta function $\zeta_q(L)$ is a systemic performance measure for all $1 \leq q \leq \infty$. Therefore by setting $q=\infty$, we have 
\[ \eta(L)=\frac{1}{2}\zeta_{\infty}(L)=\frac{1}{2}\lim_{q \rightarrow \infty} \zeta_{q}(L)~=~\frac{1}{2}\lambda_2^{-1}.\] 
As a result, $\eta(L)$ is a systemic performance measure.
\end{proof}

\hspace{.1cm}
\subsubsection{Uncertainty volume}
	The uncertainty volume of the steady-state output covariance matrix of consensus network \eqref{first-order}-\eqref{first-order-G} is defined by   
\begin{equation}
	|\Sigma|:= \det \Big( Y_{\infty} +{\frac{1}{n}J_n}\Big),
	\label{error ellipsoid}
\end{equation}
where
\[ Y_{\infty} = \lim_{t \rightarrow \infty} \mathbb{E} \big\{ y(t) y^{\text T}(t) \big\}.  \]
	This quantity is widely used as an indicator of the network performance \cite{Siami14acc, Mesbahi15cdc}. Since $y(t)$ is the error vector that represents the distance from consensus, the quantity (\ref{error ellipsoid}) is the volume of the steady-state error ellipsoid.

\begin{theorem}
\label{coro:infty}
For a given consensus network \eqref{first-order}-\eqref{first-order-G} with Laplacian matrix $L$,  the logarithm of the uncertainty volume, \ie
\begin{equation} 
	\upsilon(L) ~:= \log  |\Sigma| ~= (1-n) \log 2 - \sum_{i=2}^n \log \lambda_i
	\label{measure:uncertainty}
\end{equation}
is a systemic performance measure.
\end{theorem}

\begin{proof}
According to the dynamics of the network \eqref{first-order}-\eqref{first-order-G}, the time
evolution of the mean and the covariance matrix of the state vector are governed by
\begin{equation}
 	\dot{\bar y}(t)~=~-\left (L+\frac{1}{n}J_n \right) y(t),
 	\label{eq:834}
\end{equation}
and
\begin{equation}
	\dot Y(t)~=~-L \, Y(t)-Y(t) \, L + M_n,
	\label{eq:839}
\end{equation}
where $\bar y(t) = \E(y(t))$ and $Y(t)=\cov(y(t),y(t))$. From \eqref{eq:834}, it follows that 
\begin{equation}
	\bar y(\infty)~=~\lim_{t \rightarrow \infty} \bar y(t)~=~ 0.
	\label{eq:844}
\end{equation}
Consequently, using \eqref{eq:839} and \eqref{eq:844} we get
\begin{eqnarray*}
	Y_\infty~=~\lim_{t \rightarrow \infty} \cov(y(t),y(t)) ~=~ \frac{1}{2} L^{\dag}.
\end{eqnarray*}
Finally, by substituting $Y_{\infty}$ in \eqref{error ellipsoid}, we get 
	\begin{equation*}
		|\Sigma|= \det \Big( Y_{\infty} +{\frac{1}{n}J_n}\Big)= \det \Big( \frac{1}{2}L^{\dag} +{\frac{1}{n}J_n}\Big)= 2^{-n+1}\prod_{i=2}^n \lambda_i^{-1}.
		\label{eq:853}
	\end{equation*}
From this result and the definition of $\upsilon(L)$, one conclude  that
\begin{equation*} 
	\upsilon(L) ~=~ \log 2^{-n+1}\prod_{i=2}^n \lambda_i^{-1} ~=~ (n-1) \log 2 - \sum_{i=2}^n \log \lambda_i. 
\end{equation*}
Because $-\log (.)$ is convex and decreasing in $\R_{++}$, the quantity 
\[\upsilon(L) \,-\, (n-1) \log 2 \, = \,- \sum_{i=2}^n \log \lambda_i, \]
is a systemic performance measure according to Theorem \ref{f-sum}. Note that $(n-1) \log 2 $ is a constant number. Therefore, we conclude that $\upsilon$ is a systemic performance measure.
\end{proof}


\subsection{ Hardy-Schatten Norms of Linear Systems}
The $p$-Hardy-Schatten  norm of network \eqref{first-order}-\eqref{first-order-G} for $1 < p \leq \infty $ is defined by 
	\begin{equation}
	\|G\|_{\mathcal H_p}~:=~ \left\{ \frac{1}{2\pi} \int_{-\infty}^{\infty} \sum_{k=1}^n \sigma_k(G(j \omega))^p  \hspace{0.05cm} d\omega \right\}^{\frac{1}{p}},
	\label{h_p}
	\end{equation}
where $G(j \omega)$ is the transfer matrix of the network  from $\xi$ to $y$ and  $\sigma_k(j \omega)$ for $k =1,\ldots,n$ are singular values of $G(j \omega)$. It is known that this class of system norms captures several important performance and robustness features of linear time-invariant systems \cite{diamond2001anisotropy,vladimirov2012hardy,simon1979trace}. For example, a direct calculation shows \cite{Siami14arxiv} that the  $\mathcal H_{2}$-norm of linear consensus network \eqref{first-order}-\eqref{first-order-G} can be expressed as	
\begin{equation}
		\|G\|_{\mathcal H_2}~=~\Big(\frac{1}{2}\sum_{i=2}^n\lambda_i^{-1}\Big)^{\frac{1}{2}}.
	\end{equation}
This norm has been also interpreted as a notion of coherence in linear consensus networks \cite{Bamieh12}. The $\mathcal H_{\infty}$-norm of network \eqref{first-order}-\eqref{first-order-G} is an  input-output system norm \cite{Doyle89} and its value can be expressed as
	\begin{equation}
		\|G\|_{\mathcal H_{\infty}} ~=~ \lambda_2^{-1}, \label{H-inf}
	\end{equation}
where $\lambda_2$ is the second smallest eigenvalue of $L$, also known as the algebraic connectivity of the underlying graph of the network.  The $\mathcal H_{\infty}$-norm \eqref{H-inf} can be interpreted as the worst attainable performance against all square-integrable disturbance inputs \cite{Doyle89}. 

	\begin{theorem}
	\label{th-zeta}
The $p$-Hardy-Schatten  norm of a given consensus network \eqref{first-order}-\eqref{first-order-G} is a systemic performance measure for every exponent $2 \leq p \leq \infty$. Furthermore, the following identity holds 
	\begin{equation}
\|G\|_{\mathcal H_{p}} = \alpha_0 \big( \zeta_{p-1}(L) \big)^{1-\frac{1}{p}} \label{identity-Hp}
	\end{equation}
where $\alpha_0^{-1}=\sqrt[p]{-\beta(\frac{p}{2},-\frac{1}{2})}$ and ${\beta}:\mathbb{R}\times \mathbb{R} \rightarrow \mathbb{R}$ is the well-known Beta function.
\end{theorem}}
\begin{proof}
We utilize the disagreement form of the network that is given by \eqref{first-order-d}-\eqref{first-order-G-d} and the decomposition \eqref{G-form1} to compute the $\mathcal H_q$-norm of $G(j \omega)$ as follows \begin{eqnarray}
 \|G\|_{\mathcal H_{p}}^{p}&=&  \frac{1}{2\pi} \int_{-\infty}^{\infty}\sum_{k=1}^n \sigma_k(G(j \omega))^{p} ~d\omega \nonumber \\
&=&\frac{1}{2\pi} \sum_{i=2}^n  \int_{-\infty}^{\infty} \left(\frac{1}{\omega^2+\lambda_i^2}\right)^{\frac{p}{2}} d\omega  \nonumber \\
&=&\frac{-1}{{\beta} (\frac{p}{2},-\frac{1}{2})}\sum_{i=2}^{n}\frac{1}{\lambda_i^{p-1}} \nonumber~ = ~\frac{-1}{{\beta} (\frac{p}{2},-\frac{1}{2})}{\zeta}_{p-1}(L)^{p-1},
 \end{eqnarray}
for all $2 \leq p \leq \infty$. 
Now we show that measure \eqref{identity-Hp}  satisfies Properties 1, 2, and 3 in Definition \ref{def-schur-systemic}. Similar to the proof of Theorem \ref{f-sum}, it is straightforward to verify that measure \eqref{identity-Hp} has Property 1. Next we show that measure \eqref{identity-Hp} has Property 2, i.e., it is a convex function over the set of Laplacian matrices. 
We then show that for all $2 \leq p \leq \infty$ the following function $f: \R^{n-1}_{++} \rightarrow \R$ is concave 
\begin{equation*}
 f(x)~=~\left(\sum_{i=1}^{n-1} x_i^{-p+1}\right)^{\frac{1}{-p+1}}, 
\end{equation*} 
where $x= [x_1, x_2, \cdots, x_{n-1}]^{\text T}$.
To do so, we need to show $\triangledown^2 f(x)  \preceq 0$, where the Hessian of $f(x)$ is given by
\begin{equation*}
\frac{\partial^2 f(x)}{\partial x_i^2}~=~ -\frac{p}{x_i}\left ( \frac{f(x)}{x_i} \right)^{p} ~+~ \frac{p}{f(x)} \left( \frac{f(x)^2}{x_i^2} \right)^{p} 
\end{equation*}
and
\begin{equation*}
\frac{\partial^2 f}{\partial x_i \partial x_j}~=~ \frac{p}{f(x)} \left ( \frac{f(x)^2}{x_ix_j}\right )^{p}.
\end{equation*}
The Hessian matrix can be expressed as
\begin{equation*}
\triangledown^2 f(x) ~=~ \frac{p}{f(x)} \left( -\diag(z)^{\frac{1+p}{p}}+ zz^{\rm T}\right), 
\end{equation*}
where 
\[z~=~\left[\left ({f(x)}/{x_1}\right)^{p}, \cdots, \left ({f(x)}/{x_n}\right)^{p}\right]^{\rm T}.\] 
To verify $\triangledown^2 f(x) \preceq 0$, we must show that for all vectors $v$, $v^{\rm T} \triangledown^2 f(x) v \leq 0$. We know that 
\begin{equation}
\Scale[.92]{\displaystyle v^{\rm T} \triangledown^2 f(x) v~=~\frac{p}{f(x)} \left( -\sum_{i=1}^{n-1}z_i^{\frac{p-1}{p}}\sum_{i=1}^{n-1} z_i^{\frac{p+1}{p}} v_i^2 + \left(\sum_{i=1}^{n-1} v_i z_i \right)^2 \right).}
\end{equation}
Using the Cauchy-Schwarz inequality $a^{\rm T}b ~\leq~ \|a\|_2 \|b\|_2$, where
\begin{equation*}
a_i ~=~ \left(\frac{f(x)}{x_i}\right)^{\frac{p-1}{2}}~=~ z_i^{\frac{p-1}{2p}},
\end{equation*}
and $b_i= z_i^{\frac{p+1}{2p}} v_i$, it follows that $ v^{\rm T} \triangledown^2 f(x) v \leq 0$ for all $v \in \R^{n-1}$. Therefore, $f(x)$ is concave.
Let us define $h(x)=x^{\frac{-p+1}{p}}$, where $x \in \R$. Since $f(.)$ is positive and concave, and $h$ is decreasing convex, we conclude that $ h(f(.))$ is convex \cite{boyd2004convex}. 
Hence, we get that $\| G \|_{\mathcal H_p}$ is a convex function with respect to the eigenvalues of $L$. Since this measure is a symmetric closed convex function defined on a convex subset of $\R^{n-1}$, i.e., $n-1$ nonzero eigenvalues, according to \cite{boyd2006} we conclude that $\| G \|_{\mathcal H_p}$ is a convex of Laplacian matrix $L$.
 Finally, measure $\| G \|_{\mathcal H_p}$ is orthogonal invariant because it is a spectral function as shown in \eqref{identity-Hp}. Hence, this measure  satisfies all properties of Definition \ref{def-schur-systemic}. This completes the proof.
\end{proof}

\begin{spacing}{.88}
\bibliographystyle{IEEEtran}
\bibliography{ref/main_Milad,ref/IEEEabrv}
\end{spacing}

\vspace{-0.8cm} 

\begin{IEEEbiography}{Milad Siami} received his dual B.Sc. degrees in electrical engineering and pure mathematics from Sharif University of Technology in 2009, M.Sc. degree in electrical engineering from Sharif University of Technology in 2011, and M.Sc. and Ph.D. degrees in mechanical engineering from Lehigh University in 2014 and 2017 respectively. From 2009 to 2010, he was a research student at the Department of Mechanical and Environmental Informatics at the Tokyo Institute of Technology, Tokyo, Japan. He is currently a postdoctoral associate in the Institute for Data, Systems, and Society at MIT. His research interests include distributed control systems, distributed optimization, and applications of fractional calculus in engineering. He received a Gold Medal of National Mathematics Olympiad, Iran (2003) and the Best Student Paper Award at the 5th IFAC Workshop on Distributed Estimation and Control in Networked Systems (2015). Moreover, he was awarded  RCEAS Fellowship (2012), Byllesby Fellowship (2013), Rossin College Doctoral Fellowship (2015), and Graduate Student Merit Award (2016) at Lehigh University.
\end{IEEEbiography}

\vspace{-0.85cm} 

\begin{IEEEbiography}{Nader Motee}
(S'99-M'08-SM'13) received his B.Sc. degree in Electrical Engineering from Sharif University of Technology in 2000, M.Sc. and Ph.D. degrees from University of Pennsylvania in Electrical and Systems Engineering in 2006 and 2007 respectively. From 2008 to 2011, he was a postdoctoral scholar in the Control and Dynamical Systems Department at Caltech. He is currently an Associate Professor in the Department of Mechanical Engineering and Mechanics at Lehigh University. His current research area is distributed
dynamical and control systems with particular focus on issues related to sparsity, performance, and robustness. He is a past recipient of several awards including the 2008 AACC Hugo Schuck best paper award, the 2007 ACC best student paper award, the 2008 Joseph and Rosaline Wolf best thesis award, a 2013 Air Force Office of Scientific Research  Young Investigator Program award (AFOSR YIP), a 2015 NSF Faculty Early Career Development (CAREER) award, and a 2016 Office of Naval Research Young Investigator Program award (ONR YIP).
\end{IEEEbiography}

\end{document}